\documentclass[a4paper, reqno]{amsart}

\usepackage[utf8]{inputenc}
\usepackage{eucal, stmaryrd}
\usepackage{mathtools}
\usepackage{graphicx}
\usepackage[pdftex, colorlinks, citecolor=black, linkcolor=black,
urlcolor=black, bookmarks=false]{hyperref}
\usepackage{hyphenat}
\usepackage[table]{xcolor}

\newtheorem{lemma}{Lemma}[section]
\newtheorem{theorem}[lemma]{Theorem}

\newcommand{\N}{\mathbb{N}}
\newcommand{\Z}{\mathbb{Z}}

\newcommand{\R}{\mathbb{R}}

\newcommand{\Rplus}{\R_+}

\newcommand{\Bcal}{\mathcal{B}}
\newcommand{\Ccal}{\mathcal{C}}
\newcommand{\Ecal}{\mathcal{E}}
\newcommand{\Fcal}{\mathcal{F}}

\newcommand{\Kcal}{\mathcal{K}}

\newcommand{\Rcal}{\mathcal{R}}
\newcommand{\Scal}{\mathcal{S}}
\newcommand{\Tcal}{\mathcal{T}}

\newcommand{\sub}[2]{\mathrm{Sub}(#1, #2)}

\newcommand{\cop}{\mathrm{COP}}           % Copositive cone.
\newcommand{\copc}{\mathrm{COP}_{\rm c}}  % Copositive cone; continuous.
\newcommand{\cp}{\mathrm{CP}}             % Completely positive cone.
\newcommand{\kpcone}{\mathcal{C}}         % de Klerk-Pasechnik cone.
\newcommand{\vzcone}{\mathcal{Q}}         % Vera-Zuluaga cone.
\newcommand{\bqp}{\mathrm{BQP}}           % Boolean quadratic polytope.

\newcommand{\lass}[1]{\mathrm{lass}_{#1}}  % Lasserre hierarchy.

\newcommand{\tp}{\mathsf{T}}        % Transpose of matrix.
\newcommand{\orto}{\mathrm{O}}      % Orthogonal group.
\newcommand{\csym}{C_{\rm sym}}     % Symmetric continuous kernels.
\newcommand{\lsym}{L^2_{\rm sym}}   % Symmetric kernels in L^2.
\newcommand{\rey}[1]{\Rcal_{#1}}    % Reynolds operator.
\newcommand{\one}{\mathbf{1}}       % Constant one function.
\DeclareMathOperator{\cone}{cone}   % Conic hull.
\DeclareMathOperator{\conv}{conv}   % Convex hull.
\DeclareMathOperator{\trace}{tr}    % Trace of a matrix.
\DeclareMathOperator{\aut}{Aut}     % Automorphism group.
\newcommand{\lhs}{\mathrm{lhs}}     % Left side of a constraint.

\newcommand{\stab}[2]{{\rm Stab}(#1, #2)}

\newcommand{\flatten}[1]{\llbracket #1\rrbracket}
\newcommand{\invflatten}[2]{\llbracket #1\rrbracket_{#2}^{-1}}

\newenvironment{optprob}
{% begin code
  \arraycolsep=0pt
  \begin{array}{r@{\ }l@{\quad}l}
}%
{% end code
  \end{array}
}

\newcommand{\onerow}[1]{\multicolumn{2}{l}{#1}}

\newlength\claimlen

\newenvironment{claim}
{% begin code
  \begin{equation}%
    \setbox0=\hbox{\theequation}%
    \setlength\claimlen{\displaywidth-2\parindent}%
    \addtolength\claimlen{-\wd0}%
    \addtolength\claimlen{-9pt}%
    \begin{minipage}{\claimlen}%
}%
{% end code
  \end{minipage}%
  \end{equation}%
  \ignorespacesafterend%
}

\newcommand{\defi}[1]{\textit{#1}}

\title{Optimization hierarchies for distance-avoiding sets in compact spaces}

\author{Bram Bekker}
\address{A.J.F. Bekker, Delft Institute of Applied Mathematics,
  Delft University of Technology, Mekelweg~4, 2628~CD Delft, The
  Netherlands.}
\email{B.Bekker@tudelft.nl}

\author{Olga Kuryatnikova}
\address{O. Kuryatnikova, Department of Econometrics, Erasmus School of
  Economics, Burg.\ Oudlaan~50, 3062~PA Rotterdam, the Netherlands.}
\email{kuryatnikova@ese.eur.nl}

\author{Fernando Mário de Oliveira Filho}
\address{F.M. de Oliveira Filho, Delft Institute of Applied Mathematics,
  Delft University of Technology, Mekelweg~4, 2628~CD Delft, The
  Netherlands.}
\email{F.M.deOliveiraFilho@tudelft.nl}

\author{Juan C. Vera}
\address{J.C. Vera, Department of Econometrics and Operations Research, Tilburg
  University, Tilburg, The Netherlands.}
\email{J.C.VeraLizcano@tilburguniversity.edu}

\thanks{The first author is supported by the grant OCENW.KLEIN.024 of the Dutch
  Research Council (NWO)}


\date{September 16, 2025}

\begin{document}

\begin{abstract}
  Witsenhausen's problem asks for the maximum fraction~$\alpha_n$ of the
  $(n-1)$-dimensional unit sphere that can be covered by a measurable set containing
  no pairs of orthogonal points.  The best upper bounds for~$\alpha_n$ are given
  by extensions of the Lovász theta number.  In this paper, optimization
  hierarchies based on the Lovász theta number, like the Lasserre hierarchy, are
  extended to Witsenhausen's problem and similar problems.  These hierarchies
  are shown to converge and are used to compute the best upper bounds
  for~$\alpha_n$ in low dimensions.
\end{abstract}

\maketitle
\markboth{B. Bekker, O. Kuryatnikova, F.M. de Oliveira Filho, and
  J.C. Vera}{Optimization hierarchies for distance-avoiding sets in compact
  spaces}

\setcounter{tocdepth}{1}
\tableofcontents

%%%%%%%%%%%%%%%%%%%%%%%%%%%%%%%%%%%%%%%%%%%%%%%%%%%%%%%%%%%%%%%%%%%%%%

\section{Introduction}%
\label{sec:intro}

A set of points on the unit sphere~$S^{n-1} = \{\, x \in \R^n : \|x\|=1\,\}$
\defi{avoids orthogonal pairs} if it does not contain pairs of orthogonal vectors.
Witsenhausen's problem~\cite{Witsenhausen1974} asks for the maximum density that
a measurable subset of~$S^{n-1}$ can have if it avoids orthogonal pairs; namely,
we want to determine the parameter
\[
  \alpha_n = \sup\{\, \omega(I)/\omega_n : \text{$I \subseteq S^{n-1}$ is
    measurable and avoids orthogonal pairs}\,\},
\]
where~$\omega$ is the surface measure on the sphere and~$\omega_n$ is the total
measure of the sphere.

Denote by~$x\cdot y$ the Euclidean inner product between~$x$, $y \in \R^n$.
Fix~$e \in S^{n-1}$.  Witsenhausen~\cite{Witsenhausen1974} observed that the
union of two open antipodal spherical caps of spherical radius~$\pi/4$, namely
the set
\[
  \{\, x \in S^{n-1} : |e\cdot x| > \cos(\pi/4)\,\},
\]
avoids orthogonal pairs, hence~$\alpha_n$ is at least the density of this set,
which is $O(n^{-1/2} 2^{-n/2})$.  Kalai~\cite[Conjecture~2.8]{Kalai2015}
conjectured that this construction is optimal, that is,~$\alpha_n$ is exactly
the density of these two caps; this is known as the \defi{double-cap
  conjecture}.  A version of the double-cap conjecture for the complex sphere
has an interpretation in quantum information theory~\cite{Montina2011}.

The canonical basis vectors in~$\R^n$ are~$n$ pairwise-orthogonal unit vectors.
Any set that avoids orthogonal pairs can contain at most one of these vectors.
It then follows from a simple averaging argument that~$\alpha_n \leq 1/n$. This
upper bound was also given by Witsenhausen~\cite{Witsenhausen1974}; it is quite
far from the lower bound of the double-cap conjecture for all~$n \geq 3$.
For~$n = 2$, the lower and upper bounds coincide.  Frankl and
Wilson~\cite{FranklW1981} were the first to give an asymptotic upper bound
for~$\alpha_n$ that decreases exponentially fast with the dimension~$n$.

On the sphere, distance and inner product are related: a set of points on the
sphere avoids orthogonal pairs if and only if it avoids pairs of points at
distance~$\sqrt{2}$.  More generally, let~$V$ be a metric space with metric~$d$
and let~$D \subseteq (0, \infty)$ be a set of \defi{forbidden} distances.  We
say that a set~$I \subseteq V$ \defi{avoids} the distances in~$D$ or that it is
a \defi{$D$-avoiding set} if~$d(x, y) \notin D$ for all~$x$, $y \in I$.  One is
usually interested in the maximum size of distance-avoiding sets, which has to
be appropriately defined depending on~$V$.  Witsenhausen's problem asks for the
maximum density of a $\sqrt{2}$-avoiding set on the sphere equipped with the
Euclidean distance.

In Euclidean space, $1$-avoiding sets are of particular interest.  The maximum
density of $1$-avoiding sets in~$\R^n$ has been extensively studied (see
DeCorte, Oliveira, and Vallentin~\cite{DeCorteOV2022} for references),
especially in its relation to the chromatic number of Euclidean space.

Distance-avoiding sets can be modeled as independent sets of graphs.
Let~$G = (V, E)$ be a graph.  A subset of~$V$ is \defi{independent} if it does
not contain any edges.  The \defi{independence number} of~$G$ is
\[
  \alpha(G) = \max\{\, |I| : \text{$I \subseteq V$ is independent}\,\}.
\]
Computing the independence number of a graph is a well-known NP-hard problem.

Given a metric space~$V$ with metric~$d$ and a set~$D$ of forbidden distances,
let~$G$ be the graph with vertex set~$V$ in which~$x$, $y \in V$ are adjacent
if~$d(x, y) \in D$; we call~$G$ a \defi{distance graph}.  The independent sets
of~$G$ are exactly the $D$-avoiding sets.

Witsenhausen's problem can therefore be seen as an independent-set problem on a
distance graph over the sphere.  There are several optimization hierarchies
based on the Lovász theta number~\cite{Lovasz1979} that are used to bound the
independence number of finite graphs from above.  In this paper, we extend some
of these hierarchies from finite graphs to infinite graphs like the one related
to Witsenhausen's problem; among the hierarchies we extend is the Lasserre
hierarchy.  We prove convergence of these hierarchies and use them to compute
better upper bounds for~$\alpha_n$; see Table~\ref{tab:bounds}.

Optimization hierarchies have been used to obtain upper bounds for geometrical
parameters before; the first and perhaps most famous example is the 3-point
bound of Bachoc and Vallentin~\cite{BachocV2008} for the kissing number problem.
Hierarchies provide some of the best, and often sharp, bounds in many
circumstances.  This is the first use of optimization hierarchies to obtain
upper bounds for Witsenhausen's problem.

\begin{table}[t]
  \begin{center}
    \begin{tabular}{ccccc}
      &                    &\multicolumn{2}{c}{\textsl{Previous upper bound}}\\
      $n$&\textsl{Lower bound}&\textsl{Simple}&\textsl{Best}&\textsl{New upper bound}\\
      3 & 0.2928\ldots  & 0.3333\ldots  & 0.30153 & 0.297742\\
      4 & 0.1816\ldots  & 0.25  & 0.21676 & 0.194297\\
      5 & 0.1161\ldots  & 0.2  & 0.16765  & 0.134588\\
      6 & 0.0755\ldots  & 0.1666\ldots  & 0.13382  & 0.098095\\
      7 & 0.0498\ldots  & 0.1428\ldots  & 0.11739 & 0.075751\\
      8 & 0.0331\ldots  & 0.125 & 0.09981 & 0.061178
    \end{tabular}
  \end{center}
  \medskip

  \caption{Lower and upper bounds for Witsenhausen's parameter~$\alpha_n$.  The
    lower bound is given by the double-cap conjecture.  The simple upper bound
    was given by Witsenhausen~\cite{Witsenhausen1974} and is just~$1/n$; the
    best previous upper bounds are by DeCorte, Oliveira, and
    Vallentin~\cite{DeCorteOV2022}.  The new upper bounds are derived
    in~\S\ref{sec:computations}; see also Table~\ref{tab:detailed}.}%
  \label{tab:bounds}
\end{table}

%=====================================================================

\subsection{Infinite geometrical graphs}%
\label{sec:infinite-graphs}

Since we work with distance graphs on the sphere, it is more natural to talk
about forbidden inner products than forbidden distances.  For a
set~$D \subseteq [-1, 1]$ of forbidden inner products, denote by $G(S^{n-1}, D)$
the graph whose vertex set is~$S^{n-1}$ and in which two vertices~$x$,
$y \in S^{n-1}$ are adjacent if~$x\cdot y \in D$.

The independent sets of~$G(S^{n-1}, D)$ are precisely those sets that do not
contain pairs of points with inner product in~$D$.  For instance, the
independent sets of~$G = G(S^{n-1}, (1/2, 1))$ are exactly the sets of points on
the sphere any two of which form an angle of at least~$\pi/3$.  If we take an
independent set of~$G$ and place on each point in the set a unit sphere touching
it, we get a collection of nonoverlapping unit spheres that touch the central
unit sphere.  Conversely, any set of nonoverlapping unit spheres touching the
central unit sphere corresponds to an independent set of~$G$.  Hence the
independence number of~$G$ is precisely the \defi{kissing number} of~$\R^n$,
which is the maximum number of nonoverlapping unit spheres that can
simultaneously touch a central unit sphere.

The independent sets of the graph~$G = G(S^{n-1}, \{0\})$, which we call
\defi{Witsenhausen's graph}, are precisely the subsets of~$S^{n-1}$ that avoid
orthogonal pairs.  These sets can be infinite, and so the independence number
of~$G$ is infinite.  In this case we consider the \defi{measurable independence
  number}
\[
  \alpha_\omega(G) = \sup\{\, \omega(I) : \text{$I \subseteq S^{n-1}$ is
    measurable and independent}\,\}.
\]
Witsenhausen's parameter~$\alpha_n$ is precisely~$\alpha_\omega(G) / \omega_n$.

The graphs~$G(S^{n-1}, (1/2, 1))$ and~$G(S^{n-1}, \{0\})$ are fundamentally
different.  In the first, every vertex is contained in an open clique, and the
independence number is finite.  In the second, every vertex is contained in an
open independent set, and the independence number is infinite.

A \defi{topological graph} is a graph whose vertex set is a topological space.
The graph~$G(S^{n-1}, (1/2, 1))$ is an example of a \defi{topological packing
  graph}~\cite{LaatV2015}; these are topological graphs in which every finite
clique is a subset of an open clique.  Topological packing graphs with compact
vertex sets always have finite independence number.

The graph~$G(S^{n-1}, \{0\})$ is, in contrast, an example of a \defi{locally
  independent graph}~\cite{DeCorteOV2022}; these are topological graphs in which
every compact independent set is a subset of an open independent set.  If the
topology on~$V$ is not discrete, then locally independent graphs have infinite
independent sets, and concepts like the measurable independence number are used
instead of the independence number in these cases.  In this paper we focus on
optimization hierarchies for upper bounds for the measurable independence number
of locally independent graphs with compact vertex sets.

%=====================================================================

\subsection{Hierarchies}

The \defi{Lovász theta number} of a finite graph~$G = (V, E)$ gives an upper
bound for the independence number of~$G$.  One of its many definitions is as the
optimal value of a semidefinite programming problem:
\begin{equation}%
  \label{opt:finite-theta}
  \begin{optprob}
    \vartheta(G) = \max&\sum_{x, y \in V} A(x, y)\\
    &\sum_{x \in V} A(x, x) = 1,\\
    &A(x, y) = 0\quad\text{if~$xy \in E$,}\\
    &\text{$A\colon V^2 \to \R$ is positive semidefinite.}
  \end{optprob}
\end{equation}
Lovász~\cite{Lovasz1979} defined the theta number and showed
that~$\alpha(G) \leq \vartheta(G)$.  Indeed, if~$I \subseteq V$ is a nonempty
independent set, then the matrix~$A(x, y) = |I|^{-1} \chi_I(x) \chi_I(y)$,
where~$\chi_I$ is the characteristic vector of~$I$, is a feasible solution
of~\eqref{opt:finite-theta} with objective value~$|I|$, as we wanted.
Since~\eqref{opt:finite-theta} can be solved in polynomial
time~\cite[Theorem~9.3.30]{GrotschelLS1988}, it provides a polynomial-time
computable upper bound for the independence number.

There are basically two approaches to strengthen the Lovász theta number,
bringing it closer to the independence number.  The first approach is to
consider higher-order interactions: in the theta number, we consider
interactions between pairs of vertices; in higher-order bounds, we consider
interactions between more than two vertices at once.  The second approach is to
change the cone where~$A$ lies: in the theta number,~$A$ lies in the cone of
positive-semidefinite matrices, but we can consider subcones of this cone
instead.

The Lasserre hierarchy~\cite{Lasserre2001} is an example of the first approach.
In it, instead of considering matrices indexed by vertices, we consider matrices
indexed by subsets of vertices.  Let~$\sub{V}{k}$ denote the set of subsets
of~$V$ with cardinality at most~$k$.  We say that a
matrix~$A\colon \sub{V}{k}^2 \to \R$ is a \defi{moment matrix} if~$A(S, T)$
depends only on~$S \cup T$.  The \defi{$k$th level} of the \defi{Lasserre
  hierarchy} for~$G$ is the optimization problem
\begin{equation}%
  \label{opt:finite-lasserre}
  \begin{optprob}
    \lass{k}(G) = \max&\sum_{x \in V} A(\{x\}, \{x\})\\
    &A(\emptyset, \emptyset) = 1,\\
    &A(S, T) = 0\quad\text{if~$S \cup T$ is not independent,}\\
    &\text{$A\colon \sub{V}{k}^2 \to \R$ is a positive-semidefinite}\\
    &\qquad\text{moment matrix.}
  \end{optprob}
\end{equation}

Each level of the Lasserre hierarchy gives an upper bound for the independence
number.  Indeed, for an independent set~$I \subseteq V$ let~$A(S, T) = 1$
if~$S \cup T \subseteq I$ and~$A(S, T) = 0$ otherwise.  Then~$A$ is a feasible
solution of~\eqref{opt:finite-lasserre} with objective value~$|I|$.
Moreover,~$\lass{k}(G)$ can be computed in polynomial time for any fixed~$k$.

The first level coincides with the Lovász theta number:
$\lass{1}(G) = \vartheta(G)$.  Each following level is at least as strong as the
one before, and the hierarchy \textit{converges}, even in a finite number of
steps: if~$k \geq \alpha(G)$, then~$\lass{k}(G) = \alpha(G)$.

The completely positive formulation is an example of the second approach.  In
it, we replace in~\eqref{opt:finite-theta} the cone of positive-semidefinite
matrices by the cone of \defi{completely positive} matrices, namely
\[
  \cp(V) = \cone\{\, f f^\tp : \text{$f \in \R^V$, $f \geq 0$}\,\};
\]
call the optimal value of the resulting problem~$\vartheta(G, \cp(V))$.  The
same proof that~$\vartheta(G) \geq \alpha(G)$ given above shows that
$\vartheta(G, \cp(V)) \geq \alpha(G)$.

De Klerk and Pasechnik~\cite{KlerkP2007} observed that a theorem of Motzkin and
Straus~\cite{MotzkinS1965} implies that $\vartheta(G, \cp(V)) = \alpha(G)$.
From a computational perspective, at least at first glance, this result is not
that interesting: computing the independence number is NP-hard, and so all this
tells us is that the completely positive cone is computationally hard.  There
are several hierarchies of tractable outer approximations of the completely
positive cone, however, and using these we obtain converging
hierarchies of upper bounds for the independence number.
\medbreak
\centerline{$*$\ $*$\ $*$}
\medbreak

Bachoc, Nebe, Oliveira, and Vallentin~\cite{BachocNOV2009} observed that the
linear programming bound of Delsarte, Goethals, and Seidel~\cite{DelsarteGS1977}
is an extension of the Lovász theta number\footnote{Actually, of the
  \textit{theta prime} number, which is~\eqref{opt:finite-theta} with the extra
  constraint that~$A$ should be entrywise nonnegative.} to the topological
packing graph~$G(S^{n-1}, (1/2, 1))$.  The linear programming bound of Cohn and
Elkies for the sphere-packing problem~\cite{CohnE2003}, recently used to
determine the optimal sphere packings in dimensions~8 and~24~\cite{CohnKMRV2017,
  Viazovska2017}, is likewise an extension of the Lovász theta number to the
topological packing graph with vertex set~$\R^n$ in which~$x$, $y \in \R^n$ are
adjacent if~$0 < \|x-y\| < 1$.

Bachoc and Vallentin~\cite{BachocV2008} were the first to consider higher-order
bounds for geometric problems: their \textit{3-point bound} for the kissing
number problem can be seen~\cite{LaatMOV2022} as an intermediate step between
the first and second levels of the Lasserre hierarchy.  Later, De Laat and
Vallentin~\cite{LaatV2015} extended the Lasserre hierarchy to compact
topological packing graphs and proved convergence; more recently, De Laat, De
Muinck Keizer, and Machado~\cite{LaatMM2022} computed higher levels of the
Lasserre hierarchy for specific topological packing graphs arising from the
equiangular-lines problem.  De Laat, Machado, Oliveira, and
Vallentin~\cite{LaatMOV2022} introduced a \textit{$k$-point bound} for
topological packing graphs that lies between the steps of the Lasserre hierarchy
and that for~$k = 3$ reduces to the 3-point bound of Bachoc and Vallentin.

Cohn, De Laat, and Salmon~\cite{CohnLS2022} introduced a 3-point bound for the
sphere-packing problem, extending the 3-point bound of Bachoc and Vallentin to
topological packing graphs on the Euclidean space, which is not compact.  Cohn
and Salmon~\cite{CohnS2021} extended the Lasserre hierarchy to these graphs.

Dobre, Dür, Frerick, and Vallentin~\cite{DobreDFV2016} defined an
infinite-dimensional analogue of the copositive cone, which is the dual of the
completely positive cone, and used it to give an exact formulation for the
independence number of topological packing graphs.  In this way, they extended
the result of Motzkin and Straus~\cite{MotzkinS1965} mentioned above.
Kuryatnikova and Vera~\cite{Kuryatnikova2019} defined a hierarchy of inner
approximations of the infinite-dimensional copositive cone and used it to give a
converging hierarchy of upper bounds for the independence number of topological
packing graphs.

The picture for topological packing graphs is thus quite complete: on the
theoretical side, we have appropriate extensions of the Lasserre hierarchy and
of the completely positive/copositive formulations for such graphs, both in the
compact and noncompact settings.  On the practical side, we can use some of these
formulations to obtain better bounds in many cases of interest.

For locally independent graphs, the picture is less clear.  Bachoc, Nebe,
Oliveira, and Vallentin~\cite{BachocNOV2009} extended the Lovász theta number to
locally independent graphs on the sphere, like~$G(S^{n-1}, \{0\})$, and used it
to compute new bounds for the measurable chromatic number of Euclidean
space\footnote{The \defi{measurable chromatic number} of~$\R^n$ is the minimum
  number of colors needed to color the points of~$\R^n$ so no two points at
  distance~1 have the same color and so the sets of points having a same color
  are Lebesgue measurable.}.  Oliveira and Vallentin~\cite{OliveiraV2010} then
extended the theta number to locally independent graphs on~$\R^n$, namely to the
\defi{unit-distance graph} whose vertex set is~$\R^n$ and in which~$x$,
$y \in \R^n$ are adjacent if~$\|x-y\|=1$, and obtained better lower bounds for
the measurable chromatic number.

DeCorte, Oliveira, and Vallentin~\cite{DeCorteOV2022} extended the completely
positive formulation to locally independent graphs on compact spaces and on
Euclidean space, and proved that the formulation is exact, that is, that it
gives exactly the independence number.  The best upper bounds for Witsenhausen's
parameter~$\alpha_n$ are obtained by this technique.

In this paper, we extend the Lasserre hierarchy and the related $k$-point bound
hierarchy of De Laat, Machado, Oliveira, and Vallentin~\cite{LaatMOV2022} to
locally independent graphs on compact spaces.  We also extend the
copositive/completely positive hierarchy of Kuryatnikova and
Vera~\cite{Kuryatnikova2019}, and show that it converges for many locally
independent graphs, and in particular for Witsenhausen's
graph~$G(S^{n-1}, \{0\})$.  By comparing the $k$-point bound and the Lasserre
hierarchy to the completely positive hierarchy, we show that they also converge.

This is the first time a converging optimization hierarchy is proposed for
locally independent graphs.  We use the completely positive hierarchy to compute
the best upper bounds for Witsenhausen's parameter~$\alpha_n$ known to date; see
Table~\ref{tab:bounds} and the more detailed Table~\ref{tab:detailed}.

%%%%%%%%%%%%%%%%%%%%%%%%%%%%%%%%%%%%%%%%%%%%%%%%%%%%%%%%%%%%%%%%%%%%%%

\section{Preliminaries}

Independent sets and the independence number of a graph are defined
in~\S\ref{sec:intro}.  A \defi{topological graph} is a graph~$G = (V, E)$
where~$V$ is a topological space.  When~$G$ is a topological graph, we see~$E$
as a symmetric subset of~$V \times V$.  Topological packing graphs and locally
independent graphs are defined in~\S\ref{sec:infinite-graphs}.

Let~$G = (V, E)$ be a graph where~$V$ is a measure space with measure~$\omega$.
The \defi{measurable independence number} of~$G$ is
\[
  \alpha_\omega(G) = \sup\{\, \omega(I) : \text{$I \subseteq V$ is independent
    and measurable}\,\}.
\]
We denote by~$\aut(G)$ the group of automorphisms (that is, edge-preserving
bijections~$V \to V$) of~$G$.

Given a set~$S$ and~$X \subseteq S$, we denote by~$\chi_X\colon S \to \R$ the
characteristic function of~$X$, which is such that~$\chi_X(x) = 1$ if~$x \in X$
and~$\chi_X(x) = 0$ otherwise.

We denote by~$\one$ the constant-one function or vector and by~$J$ the
constant-one kernel (see below) or matrix.

%=====================================================================

\subsection{Functional analysis}%
\label{sec:func-analysis}

All function spaces we consider are real valued.  Let~$V$ be a measure space
equipped with a measure~$\omega$.  Unless otherwise noted, we always
equip~$L^2(V)$ with the inner product
\[
  \langle f, g\rangle = \int_V f(x) g(x)\, d\omega(x).
\]
The~$L^p$ norm of a function~$f$ is denoted by~$\|f\|_p$.  If we say that a
measurable function is nonnegative, we mean that it is nonnegative almost
everywhere.

The functions in~$L^2(V^2)$ are called \defi{kernels}.  A
kernel~$K \in L^2(V^2)$ corresponds to the operator~$K\colon L^2(V) \to L^2(V)$
such that
\[
  (Kf)(x) = \int_V K(x, y) f(y)\, d\omega(y)
\]
for all~$f \in L^2(V)$.  The operator~$K$ is self adjoint if and only if~$K$ is
\defi{symmetric}, that is,~$K(x, y) = K(y, x)$ almost everywhere.  We denote the
set of symmetric kernels by~$\lsym(V)$.

We say that a kernel~$K \in \lsym(V)$ is \defi{positive semidefinite}
if~$\langle Kf, f\rangle \geq 0$ for all~$f \in L^2(V)$.
Bochner~\cite{Bochner1941} gives a useful characterization of continuous
positive semidefinite kernels that we use throughout.  Namely, if~$V$ is a
compact Hausdorff space and if~$\omega$ is a Radon measure that is positive on
open sets, then~$K$ is positive semidefinite if and only
if~$\bigl(K(x, y)\bigr)_{x,y \in U}$ is a positive-semidefinite matrix for every
finite set~$U \subseteq V$.

We denote by~$C(V)$ the space of real-valued continuous functions on~$V$ and
by~$\csym(V)$ the space of continuous symmetric kernels on~$V^2$.

Given functions~$f\colon U \to \R$ and~$g\colon V \to \R$, we denote
by~$f\otimes g$ the function from~$U \times V$ to~$\R$ such
that~$(f\otimes g)(x, y) = f(x) g(y)$.

There are two natural topologies to consider on~$L^2(V)$.  The first is the
\defi{weak topology}.  In this topology, a net~$(f_\alpha)$ in~$L^2(V)$
converges to~$f \in L^2(V)$ if $(\langle f_\alpha, g\rangle)$ converges
to~$\langle f, g\rangle$ for all~$g \in L^2(V)$.  The second is the
\defi{$L^2$-norm topology} given by the~$L^2$ norm~$\|\cdot\|_2$.  The set of
continuous linear functionals is the same for both these topologies.  As a
consequence, if~$S \subseteq L^2(V)$ is a convex set, then its closure is the
same whether taken in the weak topology or the $L^2$-norm
topology~\cite[Theorem~5.2(iv)]{Simon2011}.

%=====================================================================

\subsection{Spaces of subsets}

Let~$V$ be a set.  For an integer~$k \geq 0$, denote by~$\sub{V}{k}$ the
collection of all subsets of~$V$ with cardinality at most~$k$.  For~$k \geq 1$
and~$v = (v_1, \ldots, v_k) \in V^k$, denote by~$\flatten{v}$ the
set~$\{v_1, \ldots, v_k\}$.  So, for every~$k \geq 1$ we have
that~$\flatten{\cdot}$ maps~$V^k$ to~$\sub{V}{k}$.  Note that~$k$ is superfluous
in the definition of~$\flatten{\cdot}$, but not in the definition of the
preimage.  Hence, given a set~$S \subseteq \sub{V}{k}$, we write
\[
  \invflatten{S}{k} = \{\, v \in V^k : \flatten{v} \in S\,\}.
\]

If~$V$ is a topological space, then we can introduce
on~$\sub{V}{k} \setminus \{\emptyset\}$ the quotient topology
of~$\flatten{\cdot}$ by declaring a set~$S$ open if~$\invflatten{S}{k}$ is open
in~$V^k$.  We define a topology on~$\sub{V}{k}$ by taking the disjoint union
with~$\{\emptyset\}$.  For background on the topology on the space of subsets,
see Handel~\cite{Handel2000}.

If~$V$ is a measure space equipped with a measure~$\omega$, then we can
turn $\sub{V}{k} \setminus \{\emptyset\}$ into a measure space by considering
the pushforward of~$\omega$ through~$\flatten{\cdot}$.  Namely, we declare a set~$S$
measurable if~$\invflatten{S}{k}$ is measurable in~$V^k$ with respect to the
product measure, and we set the measure of~$S$ to be the measure
of~$\invflatten{S}{k}$.  We always define the measure on~$\sub{V}{k}$ by setting
the measure of~$\{\emptyset\}$ to be~1, that is, if~$f\colon \sub{V}{k} \to \R$
is a continuous function, then
\[
  \int_{\sub{V}{k}} f(S)\, d\omega(S) = f(\emptyset) + \int_{V^k}
  f(\flatten{v})\, d\omega(v).
\]

%=====================================================================

\subsection{Topological groups}%
\label{sec:topo-groups}

A \defi{topological group} is a group~$\Gamma$ equipped with a topology in which
the group operations --- multiplication and inversion --- are continuous, that
is, $(\sigma, \tau) \mapsto \sigma \tau$ is a continuous function
from~$\Gamma \times \Gamma$ to~$\Gamma$ and~$\sigma \mapsto \sigma^{-1}$ is a
continuous function from~$\Gamma$ to~$\Gamma$.  We denote the identity element
by~$1$.  For us, topological groups are always Hausdorff spaces.

Let~$\Gamma$ be a locally compact group and~$V$ be a locally compact Hausdorff
space.  An \defi{action} of~$\Gamma$ on~$V$ is a continuous
map~$(\sigma, x) \mapsto \sigma x$ from~$\Gamma \times V$ to~$V$ such that
(i)~$x\mapsto \sigma x$ is a homeomorphism of~$V$ for every~$\sigma \in \Gamma$
and (ii)~$\sigma(\tau x) = (\sigma\tau)x$ for all~$\sigma$, $\tau \in \Gamma$
and~$x \in V$.  We call~$V$ a \defi{$\Gamma$-space} when there is an action
of~$\Gamma$ on~$V$.

The action is called \defi{transitive}, and~$V$ is called a \defi{transitive
  $\Gamma$-space}, if for every~$x$, $y \in V$ there is~$\sigma \in \Gamma$ such
that~$\sigma x = y$.  Let~$V$ be a transitive $\Gamma$-space and fix~$e \in V$;
consider the map~$p\colon \Gamma \to V$ such that~$p(\sigma) = \sigma e$.  Note
that~$p$ is surjective; if it is also open, mapping open subsets of~$\Gamma$ to
open subsets of~$V$, then we call~$V$ a \defi{homogeneous $\Gamma$-space}.
If~$\Gamma$ is $\sigma$-compact, then~$p$ is
open~\cite[Proposition~2.44]{Folland1995}.

If~$\Gamma$ acts on~$V$, then it also acts on functions~$f\colon V^k \to \R$ by
\[
  (\sigma f)(x_1, \ldots, x_k) = f(\sigma^{-1} x_1, \ldots, \sigma^{-1}
  x_k).
\]
We always consider this action on function spaces, unless otherwise noted.  We
say that a function~$f\colon V^k \to \R$ is \defi{$\Gamma$-invariant} (or simply
\defi{invariant} when the group is clear from context) if~$\sigma f = f$ for
all~$\sigma \in \Gamma$.

There is a useful relation between functions on~$\Gamma^k$ and functions
on~$V^k$ when~$V$ is a homogeneous $\Gamma$-space.  Namely, fix~$e \in V$.  From
a function~$f\colon V^k \to \R$ we can construct a
function~$g\colon \Gamma^k \to \R$ by setting
\[
  g(\sigma_1, \ldots, \sigma_k) = f(\sigma_1 e, \ldots, \sigma_k e).
\]
If~$f$ is continuous, then so is~$g$.  Moreover,
\begin{claim}%
  \label{claim:homo-func}
  if~$\sigma_i e = \tau_i e$ for~$i = 1$, \dots,~$k$, then~$g(\sigma_1,
  \ldots, \sigma_k) = g(\tau_1, \ldots, \tau_k)$.
\end{claim}

Conversely, given a function~$g$ satisfying~\eqref{claim:homo-func}, we can
define~$f\colon V^k \to \R$ by setting
\[
  f(x_1, \ldots, x_k) = g(\sigma_1, \ldots, \sigma_k),
\]
where the~$\sigma_i$ are any elements of~$\Gamma$ such that~$\sigma_i e = x_i$.
Like before, if~$g$ is continuous, then~$f$ is also continuous.  Indeed, given
an open set~$A \subseteq \R$, we want to show that~$f^{-1}(A)$ is open
in~$V^k$.  Consider the map~$p^k\colon \Gamma^k \to V^k$ such that
$p^k(\sigma_1, \ldots, \sigma_k) = (\sigma_1 e, \ldots, \sigma_k e)$.  Note that
$f^{-1}(A) = p^k(g^{-1}(A))$; since~$p^k$ is an open map, we are done.

Let~$\Gamma$ be a compact group.  We always normalize the Haar measure~$\mu$
on~$\Gamma$ so~$\mu(\Gamma) = 1$.  For a compact group, the Haar measure is both
left and right invariant, that is,~$\mu(\sigma X \tau) = \mu(X)$ for
all~$\sigma$, $\tau \in \Gamma$ and measurable~$X \subseteq \Gamma$; moreover,
the measure~$\mu_{-1}$ such
that~$\mu_{-1}(X) = \mu(\{\,\sigma^{-1} : \sigma \in X\,\})$ coincides
with~$\mu$.

If~$V$ is a homogeneous $\Gamma$-space, then~$V$ is in particular compact.  The
pushforward of~$\mu$ to~$V$ is the Radon measure~$\omega$ such that
\[
  \int_V f(x)\, d\omega(x) = \int_\Gamma f(\sigma e)\, d\mu(\sigma)
\]
for every integrable function~$f\colon V \to \R$ and every~$e \in V$.  This is a
$\Gamma$-invariant measure: $\omega(\sigma X) = \omega(X)$ for
all~$\sigma \in \Gamma$ and measurable~$X \subseteq V$; it is moreover the
unique $\Gamma$-invariant Radon measure on~$V$, up to a constant
factor~\cite[Theorem~2.49]{Folland1995}.

Say~$\Gamma$ is a compact group acting on a compact Hausdorff space~$V$;
let~$\mu$ be the Haar measure on~$\Gamma$.  The \defi{Reynolds operator} maps a
measurable function~$f\colon V \to \R$ to the $\Gamma$-invariant
function~$\rey{\Gamma} f$ such that
\[
  (\rey{\Gamma} f)(x) = \int_\Gamma f(\sigma x)\, d\mu(\sigma).
\]
When using the Reynolds operator, the action of~$\Gamma$ will always be clear
from the context.  Finally, if~$V$ is equipped with a finite measure invariant
under the action of~$\Gamma$, then the Reynolds operator is self adjoint, that
is, $\langle \rey{\Gamma} f, g\rangle = \langle f, \rey{\Gamma} g\rangle$.

%%%%%%%%%%%%%%%%%%%%%%%%%%%%%%%%%%%%%%%%%%%%%%%%%%%%%%%%%%%%%%%%%%%%%%

\section{The Lasserre hierarchy and the \texorpdfstring{$k$}{k}-point bound}%
\label{sec:lass-k-point}

Let~$G = (V, E)$ be a topological graph and~$\omega$ be a Borel measure on~$V$.
For an integer~$k \geq 2$, let~$M\colon C(\sub{V}{2k}) \to C(\sub{V}{k}^2)$ be
the operator such that
\[
  (M\nu)(S, T) = \nu(S \cup T).
\]
Note that~$M\nu$ is indeed continuous since the union map
$(S, T) \mapsto S \cup T$ is continuous~\cite[Proposition~2.14]{Handel2000}.
The \defi{$k$th level} of the \defi{Lasserre hierarchy} for~$G$ is the
optimization problem
\[
  \begin{optprob}
    \lass{k}(G) = \sup&\onerow{\int_V \nu(\{x\}) d\omega(x)}\\
    &\nu(\emptyset) = 1,\\
    &\nu(S) = 0&\text{if~$S \in \sub{V}{2k}$ is not independent,}\\
    &\onerow{\text{$M\nu$ is positive semidefinite,}}\\
    &\onerow{\nu \in C(\sub{V}{2k}).}
  \end{optprob}
\]
Depending on~$G$, this problem could be infeasible.  We denote by~$\lass{k}(G)$
both the optimal value of the problem above and the problem itself, and we
follow the same convention for other optimization problems below.

For an integer~$k \geq 2$ and a set~$Q \subseteq V$ with~$|Q| \leq k - 2$,
let~$M_Q\colon C(\sub{V}{k}) \to C(\sub{V}{1}^2)$ be the operator such that
\[
  (M_Q\nu)(S, T) = \nu(Q \cup S \cup T);
\]
again,~$M_Q\nu$ is indeed continuous.  The \defi{$k$-point bound} for~$G$ is the
optimization problem
\begin{equation}%
  \label{opt:k-point}
  \begin{optprob}
    \Delta_k(G) = \sup&\onerow{\int_V \nu(\{x\})\, d\omega(x)}\\
    &\nu(\emptyset) = 1,\\
    &\nu(S) = 0&\text{if~$S \in \sub{V}{k}$ is not independent,}\\
    &\onerow{\text{$M_Q\nu$ is positive semidefinite for every independent}}\\
    &\onerow{\phantom{M_q\nu\ \ }\text{set~$Q \in \sub{V}{k-2}$,}}\\
    &\onerow{\nu \in C(\sub{V}{k}).}
  \end{optprob}
\end{equation}

The restriction of a feasible solution of~$\lass{k+1}(G)$ to~$\sub{V}{2k}$ is
continuous~\cite[Proposition~2.4]{Handel2000}, hence it is a feasible
solution of~$\lass{k}(G)$.  The same can be said about~$\Delta_{k+1}(G)$
and~$\Delta_k(G)$, and so we have
\[
  \lass{1}(G) \geq \lass{2}(G) \geq \cdots\qquad\text{and}\qquad
  \Delta_2(G) \geq \Delta_3(G) \geq \cdots.
\]
Moreover, it is immediate that~$\lass{k}(G) \leq \Delta_{k+1}(G)$, since
if~$\nu \in C(\sub{V}{2k})$ is a feasible solution of~$\lass{k}(G)$, then for
every independent set~$Q \subseteq V$ with~$|Q| \leq k-1$ and every~$S$,
$T \in \sub{V}{1}$ we have~$(M_Q\nu')(S, T) = (M\nu)(Q \cup S, Q \cup T)$,
where~$\nu'$ is the restriction of~$\nu$ to~$\sub{V}{k+1}$; it follows
that~$M_Q\nu'$ is positive semidefinite.

When~$V$ is finite with the discrete topology and the counting
measure,~$\lass{k}(G)$ is simply the Lasserre
hierarchy~\eqref{opt:finite-lasserre} for the independence number,
whereas~$\Delta_k(G)$ is closely related to restrictions of the Lasserre
hierarchy proposed by Gvozdenović, Laurent, and
Vallentin~\cite{GvozdenovicLV2009} and De Laat, Machado, Oliveira, and
Vallentin~\cite{LaatMOV2022}.  In this case, we
have~$\lass{k}(G) \geq \alpha(G)$ and~$\Delta_k(G) \geq \alpha(G)$ for all~$k$.

The proof is simple.  Given an independent set~$I \subseteq V$, let~$\nu(S) = 1$
if~$S \subseteq I$ and~$\nu(S) = 0$ otherwise for every~$S \in \sub{V}{2k}$.
Then~$\nu(\emptyset) = 1$ and~$\nu(S) = 0$ if~$S$ is not independent.  Moreover,
since~$\nu(\{x_1, \ldots, x_t\}) = \chi_I(x_1) \cdots \chi_I(x_t)$, we see
that~$M\nu$ is positive semidefinite.  So~$\nu$ is a feasible solution
of~$\lass{k}(G)$ with objective value~$|I|$, and since~$I$ is any independent
set we get~$\lass{k}(G) \geq \alpha(G)$.  In the same way one shows
that~$\Delta_k(G) \geq \alpha(G)$.

This same proof fails in general for infinite topological graphs, since the
function~$\nu$ constructed above will often not be continuous.  This issue can
be overcome with extra assumptions.

\begin{theorem}%
  \label{thm:upper-bound}
  Let~$G = (V, E)$ be a topological graph and let~$\Gamma\subseteq\aut(G)$ be a
  compact group.  If~$V$ is a homogeneous $\Gamma$-space and if~$\omega$ is the
  pushforward to~$V$ of the Haar measure on~$\Gamma$,
  then~$\lass{k}(G) \geq \alpha_\omega(G)$
  and~$\Delta_k(G) \geq \alpha_\omega(G)$ for all~$k$.
\end{theorem}

The main step in the proof uses the following lemma.

\begin{lemma}%
  \label{lem:tensor-reynolds}
  Let~$\Gamma$ be a compact group and let~$V$ be a homogeneous
  $\Gamma$-space equipped with the pushforward of the Haar measure~$\mu$
  on~$\Gamma$.  If~$f_1$, \dots,~$f_k \in L^k(V)$, then the function
  \[
    \rey{\Gamma}(f_1 \otimes \cdots \otimes f_k)(x_1, \ldots, x_k) =
    \int_\Gamma f_1(\alpha x_1)\cdots f_k(\alpha x_k)\, d\mu(\alpha)
  \]
  from~$V^k$ to~$\R$ is continuous.
\end{lemma}

\begin{proof}
In view of~\S\ref{sec:topo-groups}, it suffices to consider~$V = \Gamma$.
If~$f_1$, \dots,~$f_k$ are continuous, then they are left and right uniformly
continuous~\cite[Proposition~2.6]{Folland1995}, and the result is immediate.

Thus assume~$f_i \in L^k(\Gamma)$ and fix~$\epsilon > 0$.  Since continuous
functions are dense in~$L^k(\Gamma)$, let~$g_i \in C(\Gamma)$ be such
that~$\|f_i - g_i\|_k < \epsilon$.  Our goal is to find an upper bound for
\begin{equation}%
  \label{eq:sym-goal}
  \begin{split}
    &|\rey{\Gamma}(f_1 \otimes \cdots \otimes f_k)(\sigma_1, \ldots, \sigma_k)
    - \rey{\Gamma}(g_1 \otimes \cdots \otimes g_k)(\sigma_1, \ldots,
    \sigma_k)|\\
    &\qquad=\biggl|\int_\Gamma f_1(\alpha\sigma_1) \cdots f_k(\alpha\sigma_k)
    - g(\alpha\sigma_1) \cdots g(\alpha\sigma_k)\, d\mu(\alpha)\biggr|
  \end{split}
\end{equation}
that is uniform on~$\sigma_1$, \dots,~$\sigma_k \in \Gamma$.  If this bound goes
to zero with~$\epsilon$, then~$\rey{\Gamma}(f_1 \otimes \cdots \otimes f_k)$ is
a uniform limit of continuous functions, being therefore continuous, as we want.

Using the triangle inequality, we see that~\eqref{eq:sym-goal} is at most
\[
  \begin{split}
    &\int_\Gamma |f_1(\alpha\sigma_1) \cdots f_k(\alpha\sigma_k) -
    f_1(\alpha\sigma_1) \cdots
    f_{k-1}(\alpha\sigma_{k-1})g_k(\alpha\sigma_k)\\
    &\quad\qquad{}+f_1(\alpha\sigma_1) \cdots
    f_{k-1}(\alpha\sigma_{k-1})g_k(\alpha\sigma_k) - g_1(\alpha\sigma_1)
    \cdots g_k(\alpha\sigma_k)|\, d\mu(\alpha)\\
    &\quad\leq \int_\Gamma |f_1(\alpha\sigma_1) \cdots
    f_{k-1}(\alpha\sigma_{k-1})| |f_k(\alpha\sigma_k) -
    g_k(\alpha\sigma_k)|\, d\mu(\alpha)\\
    &\quad\qquad{}+ \int_\Gamma |f_1(\alpha\sigma_1) \cdots
    f_{k-1}(\alpha\sigma_{k-1}) - g_1(\alpha\sigma_1) \cdots
    g_{k-1}(\alpha\sigma_{k-1})| |g_k(\alpha\sigma_k)|\, d\mu(\alpha).
  \end{split}
\]

By repeating this procedure, we see that~\eqref{eq:sym-goal} is bounded from
above by a sum of~$k$ terms of the form
\begin{equation}%
  \label{eq:sym-goal-term}
  \int_\Gamma |f_i(\alpha\sigma_i) - g_i(\alpha\sigma_i)| |h_1(\alpha)|
  \cdots |h_{k-1}(\alpha)|\, d\mu(\alpha),
\end{equation}
where each~$h_j$ is one of the functions~$\alpha \mapsto f_l(\alpha\sigma_l)$
or~$\alpha \mapsto g_l(\alpha\sigma_l)$ for some~$l$.

A recursive application of Hölder's inequality shows that each
summand~\eqref{eq:sym-goal-term} is at most
\[
  \|f_i - g_i\|_k \|h_1\|_k \cdots \|h_{k-1}\|_k \leq \epsilon M^{k-1},
\]
where~$M = \epsilon + \max\{\|f_1\|_k, \ldots, \|f_k\|_k\}$.
So~\eqref{eq:sym-goal} is bounded from above by~$\epsilon k M^{k-1}$,
thus~$\rey{\Gamma}(f_1 \otimes \cdots \otimes f_k)$ is a uniform limit of
continuous functions and hence continuous.
\end{proof}

\begin{proof}[Proof of Theorem~\ref{thm:upper-bound}]
Fix an integer~$k \geq 1$ and let~$I \subseteq V$ be a measurable independent
set.  From Lemma~\ref{lem:tensor-reynolds} we know
that~$F = \rey{\Gamma}\chi_I^{\otimes 2k}$ is continuous.  Moreover,
since~$\chi_I$ is~$0$--$1$, we know that~$F(v)$ depends only on~$\flatten{v}$:
for all~$u$, $v \in V^{2k}$, if~$\flatten{u} = \flatten{v}$, then~$F(u) = F(v)$.
So we can define~$\nu\colon \sub{V}{2k} \to \R$ by setting~$\nu(\emptyset) = 1$
and~$\nu(\flatten{v}) = F(v)$ for all~$v \in V^{2k}$.

If we show that~$\nu$ is continuous, then it is clear that~$\nu$ is a feasible
solution of~$\lass{k}(G)$, and
so~$\lass{k}(G) \geq \int_V \nu(\{x\})\, d\omega(x) = \omega(I)$.  Since~$I$ is
any measurable independent set, it follows
that~$\lass{k}(G) \geq \alpha_\omega(G)$.

It suffices to show that the restriction~$\overline{\nu}$ of~$\nu$
to~$\sub{V}{2k}\setminus \{\emptyset\}$ is continuous.  So let~$A \subseteq \R$
be an open set; we want to show that~$\overline{\nu}^{-1}(A)$ is open.  This is
the case, by definition\footnote{Recall that for~$S \subseteq \sub{V}{n}$ we
  write $\invflatten{S}{n} = \{\, v \in V^n : \flatten{v} \in S\,\}$.},
if~$\invflatten{\overline{\nu}^{-1}(A)}{2k}$ is open in~$V^{2k}$.  Now note
that~$\invflatten{\overline{\nu}^{-1}(A)}{2k} = F^{-1}(A)$.  Since~$F$ is
continuous,~$F^{-1}(A)$ is open, and we are done.

The proof that~$\Delta_k(G) \geq \alpha_\omega(G)$ is similar.
\end{proof}

%%%%%%%%%%%%%%%%%%%%%%%%%%%%%%%%%%%%%%%%%%%%%%%%%%%%%%%%%%%%%%%%%%%%%%

\section{A completely positive hierarchy}%
\label{sec:cop-hierarchy}

Let~$V$ be a compact Hausdorff space and~$\omega$ be a Radon measure on~$V$.
The \defi{copositive cone} on~$V$ is the convex cone
\[
  \cop(V) = \{\, A \in \lsym(V) :
  \text{$\langle Af, f\rangle \geq 0$ for all~$f \in L^2(V)$
    with~$f \geq 0$}\,\}.
\]
We say that a kernel is \defi{copositive} if it belongs to~$\cop(V)$.

Note that~$\cop(V)$ is closed in the~$L^2$-norm topology.  Since the weak
topology and the~$L^2$-norm topology have the same continuous linear
functionals, and since~$\cop(V)$ is convex, it is also closed in the weak
topology (see~\S\ref{sec:func-analysis}).

The dual of~$\cop(V)$ is the \defi{completely positive cone} on~$V$, namely
\[
  \cp(V) = \cop(V)^* = \{\, A \in \lsym(V) :
  \text{$\langle A, Z\rangle \geq 0$ for all~$Z \in \cop(V)$}\,\}.
\]
We say that a kernel is \defi{completely positive} if it belongs to~$\cp(V)$.

Let~$\Kcal \subseteq \lsym(V)$ be any closed convex cone containing~$\cp(V)$.
Given a graph~$G = (V, E)$, consider the optimization problem
\[
  \begin{optprob}
    \vartheta(G, \Kcal) = \sup&\langle A, J\rangle\\
    &\int_V A(x, x)\, d\omega(x) = 1,\\
    &A(x, y) = 0\quad\text{if~$xy \in E$,}\\
    &\onerow{\text{$A \in \Kcal$ is continuous,}}
  \end{optprob}
\]
where~$J$ is the constant-one kernel.

DeCorte, Oliveira, and Vallentin~\cite{DeCorteOV2022} showed that, if~$G$ is
locally independent, then~$\vartheta(G, \Kcal) \geq \alpha_\omega(G)$.  They
also showed that $\vartheta(G, \cp(V)) = \alpha_\omega(G)$ under some extra
assumptions, extending a result of Motzkin and Straus~\cite{KlerkP2002a,
  MotzkinS1965} to locally independent graphs.  An obvious choice for~$\Kcal$ is
the cone of positive semidefinite kernels.  In this case, we recover the
extension of the Lovász theta number introduced by Bachoc, Nebe, Oliveira, and
Vallentin~\cite{BachocNOV2009}.

For an integer~$r \geq 1$, consider the set
\[
  \kpcone_r(V) = \{\, A \in \lsym(V) : \rey{\Scal_{r+2}}(A \otimes
  \one^{\otimes r}) \geq 0\, \}.
\]
Here, the symmetric group~$\Scal_{r+2}$ acts on~$V^{r+2}$ by permuting
coordinates, so for $F\colon V^{r+2} \to \R$ we have
\[
  (\rey{\Scal_{r+2}} F)(x_1, \ldots, x_{r+2}) = \frac{1}{(r+2)!} \sum_{\pi
    \in \Scal_{r+2}} F(x_{\pi(1)}, \ldots, x_{\pi(r+2)}).
\]
Note that~$\kpcone_r(V)$ is a convex cone.

For~$A \in \lsym(V)$,
write~$\Tcal_r A = \rey{\Scal_{r+2}}(A \otimes \one^{\otimes r})$.  Use the
triangle inequality to show that~$\|\rey{\Scal_{r+2}} F\|_2 \leq \|F\|_2$
for~$F \in L^2(V^2)$; it follows that~$\Tcal_r A \in L^2(V^{r+2})$ for
all~$A \in \lsym(V)$, and so~$\Tcal_r$ is a linear transformation
from~$\lsym(V)$ to~$L^2(V^{r+2})$.

Consider the linear transformation~$\Tcal_r^*\colon L^2(V^{r+2}) \to \lsym(V)$
such that
\[
  (\Tcal_r^* F)(x, y) = \int_{V^r} (\rey{\Scal_{r+2}} F)(x, y, v)\,
  d\omega(v)
\]
for~$F \in L^2(V^{r+2})$ and~$x$, $y \in V$, where~$F(x, y, v)$ is short
for~$F(x, y, v_1, \ldots, v_r)$ and where we consider on~$V^r$ the product
measure.  One easily verifies that~$\Tcal_r^* F \in \lsym(V)$.  Direct
computation shows that
$\langle \Tcal_r A, F\rangle = \langle A, \Tcal_r^* F\rangle$.

Both~$\Tcal_r$ and~$\Tcal_r^*$ are continuous in the weak topology (and also in
the norm topology, though we do not need this fact).  Indeed, if~$(A_\alpha)$ is
a net in~$\lsym(V)$ that converges to~$A$ and if~$F \in L^2(V^{r+2})$, then
\[
  \langle \Tcal_r A_\alpha, F\rangle = \langle A_\alpha, \Tcal_r^* F\rangle
  \to \langle A, \Tcal_r^* F\rangle = \langle \Tcal_r A, F\rangle,
\]
and we see that~$\Tcal_r A_\alpha \to \Tcal_r A$.  In the same way we show
that~$\Tcal_r^*$ is continuous.  It follows that~$\Tcal_r^*$ is the adjoint
of~$\Tcal_r$.

The following theorem was shown by Kuryatnikova and Vera~\cite{Kuryatnikova2019}
for the continuous counterpart of~$\kpcone_r(V)$.

\begin{theorem}%
  \label{thm:kpcone-hierarchy}
  Let~$V$ be a compact Hausdorff space and let~$\omega$ be a Radon measure
  on~$V$.  The cone~$\kpcone_r(V)$ is closed for every~$r \geq 1$, both in
  the~$L^2$-norm topology and the weak topology, and
  \[
    \kpcone_1(V) \subseteq \kpcone_2(V) \subseteq \cdots \subseteq \cop(V).
  \]
\end{theorem}

The proof requires the following lemma.

\begin{lemma}%
  \label{lem:nonneg-tensor}
  Let~$V$ be a Hausdorff space and~$\omega$ be a Radon measure on~$V$.  A
  function~$F \in L^2(V^k)$ is nonnegative if and only
  if~$\langle F, g_1 \otimes \cdots \otimes g_k\rangle \geq 0$ for all
  nonnegative~$g_1$, \dots,~$g_k \in L^2(V)$.
\end{lemma}

\begin{proof}
Necessity is easy, so let us prove sufficiency.

Since~$F$ is~$L^2$, if it is not nonnegative, then there is a set~$X \subseteq
V^k$ of finite positive measure such that~$\langle F, \chi_X\rangle < 0$.
Since~$\omega$ is outer regular, we can approximate~$X$ in measure arbitrarily
well by open sets containing~$X$, and so by using the Cauchy-Schwarz inequality
we see that there is an open set~$A$ of finite measure such that~$\langle F,
\chi_A\rangle < 0$.

% Proof that there is a set X of finite measure etc.
% Take X = { x in V : F(x) < 0 }.  Now, int(F(x)^2, x in X) < infty, since F is
% L^2.  Write U_k = {x in X : F(x)^2 > 1 / k}.  Note that X is the union of the
% U_k (a countable union), and that each U_k has finite measure, otherwise the
% total integral is infinite, a contradiction.  Also, there must be a U_k with
% positive measure, or X itself has measure 0, and we are done.

Now~$\omega$ is also inner regular on open sets.  So fix~$\epsilon > 0$ and
let~$C \subseteq A$ be a compact set such
that~$\omega(A \setminus C) < \epsilon$.  Sets of the
form~$S_1 \times \cdots \times S_k$, where~$S_1$, \dots,~$S_k \subseteq V$ are
open, are a base of the product topology on~$V^k$.  Together with the
compactness of~$C$, this means that there is a finite cover~$\Ccal$ of~$C$ by
sets of the form~$S_1 \times \cdots \times S_k \subseteq A$, where~$S_1$,
\dots,~$S_k \subseteq V$ are open.

Let~$\Bcal = \{\, S_i : \text{$S_1 \times \cdots \times S_k \in \Ccal$
  and~$i = 1$, \dots,~$k$}\,\}$.  For~$\Scal \subseteq \Bcal$, write
\begin{equation}%
  \label{eq:part-def}
  E_\Scal = \bigcap_{S \in \Scal} S \cap \bigcap_{S \in
    \Bcal\setminus\Scal} V\setminus S
\end{equation}
and~$\Ecal = \{\, E_\Scal : \text{$\Scal \subseteq \Bcal$
  and~$E_\Scal \neq \emptyset$}\,\}$.  Note that the sets in~$\Ecal$ are
pairwise disjoint and that sets of the form~$T_1 \times \cdots \times T_k$
with~$T_i \in \Ecal$ cover~$C$.

Now let~$\Fcal = \{\, T_1 \times \cdots \times T_k : \text{$T_i \in \Ecal$
  and~$T_1 \times \cdots \times T_k \cap C \neq \emptyset$}\,\}$, so~$\Fcal$ is
a cover of~$C$.  We claim that every set in~$\Fcal$ is contained in~$A$.
Indeed, if~$T_1 \times \cdots \times T_k \in \Fcal$, then there
is~$S_1 \times \cdots \times S_k \in \Ccal$ such that
$T_1 \times \cdots \times T_k \cap S_1 \times \cdots \times S_k \neq \emptyset$.
This implies that~$T_i \cap S_i \neq \emptyset$ for all~$i$.  But since~$T_i$ is
of the form~\eqref{eq:part-def}, if~$T_i \cap S_i \neq \emptyset$,
then~$T_i \subseteq S_i$, and so
$T_1 \times \cdots \times T_k \subseteq S_1 \times \cdots \times S_k \subseteq
A$.

Setting~$C' = \bigcup \Fcal$ we get
\[
  \begin{split}
    \langle F, \chi_{C'}\rangle &= \langle F, \chi_C\rangle + \langle F,
    \chi_{C'\setminus C}\rangle\\
    &\leq \langle F, \chi_C\rangle + \|F\|_2\|\chi_{C'\setminus C}\|_2\\
    &\leq \langle F, \chi_C\rangle + \|F\|_2\|\chi_{A\setminus C}\|_2\\
    &< \langle F, \chi_C\rangle + \|F\| \epsilon^{1/2}.
  \end{split}
\]
By taking~$\epsilon \to 0$ we
get~$\langle F, \chi_C\rangle \to \langle F, \chi_A\rangle < 0$, and so for
small enough~$\epsilon$ we have~$\langle F, \chi_{C'}\rangle < 0$.  Since the
sets in~$\Fcal$ are disjoint, there must then
be~$T_1 \times \cdots \times T_k \in \Fcal$ such
that~$\langle F, \chi_{T_1 \times \cdots \times T_k}\rangle < 0$.  But then
setting~$g_i = \chi_{T_i}$ we see that~$g_i \geq 0$ and
that~$\langle F, g_1 \otimes \cdots \otimes g_k\rangle < 0$.
\end{proof}

\begin{proof}[Proof of Theorem~\ref{thm:kpcone-hierarchy}]
To see that~$\kpcone_r(V)$ is weakly closed, note that $\kpcone_r(V) = \{\, A
\in \lsym(V) : \Tcal_r A \geq 0\,\}$.  If~$(A_\alpha)$ is a net
in~$\kpcone_r(V)$ converging to~$A$, then since~$\Tcal_r$ is weakly continuous
we know that~$\Tcal_r A_\alpha \to \Tcal_r A$.  Finally, since the set of
nonnegative functions in~$L^2(V^{r+2})$ is weakly closed, we see that~$\Tcal_r A
\geq 0$, and so~$A \in \kpcone_r(V)$.  Since~$\kpcone_r(V)$ is convex and since
the~$L^2$-norm topology and the weak topology have the same continuous
functionals, it follows that~$\kpcone_r(V)$ is closed in the norm topology as
well (see~\S\ref{sec:func-analysis}).

We now show the chain of inclusions.  Given a kernel~$A \in \kpcone_r(V)$ for
some~$r$, we want to show that~$A \in \kpcone_{r+1}(V)$, that is, we want to
show $\rey{\Scal_{r+3}}(A \otimes \one^{\otimes (r+1)}) \geq 0$, and for
that we use Lemma~\ref{lem:nonneg-tensor}.  So let~$f_1$,
\dots,~$f_{r+3} \in L^2(V)$ be any nonnegative functions.  For~$i = 1$,
\dots,~$r+3$, let~$g_i$ be the tensor product of the functions~$f_1$,
\dots,~$f_{r+3}$, except for~$f_i$, in some arbitrary order.  Then
\[
  \begin{split}
    &\langle \rey{\Scal_{r+3}}(A \otimes \one^{\otimes(r+1)}), f_1 \otimes
    \cdots \otimes f_{r+3}\rangle\\
    &\qquad=\langle A \otimes 1^{\otimes(r+1)},
    \rey{\Scal_{r+3}}(f_1 \otimes \cdots \otimes f_{r+3})\rangle\\
    &\qquad=\frac{1}{(r+3)!} \sum_{\pi \in \Scal_{r+3}} \langle A \otimes
    \one^{\otimes(r+1)}, f_{\pi(1)} \otimes \cdots \otimes
    f_{\pi(r+3)}\rangle\\
    &\qquad=\frac{1}{r+3} \sum_{i=1}^{r+3} \langle A \otimes \one^{\otimes
      r} \otimes \one, \rey{\Scal_{r+2}} g_i \otimes f_i\rangle\\
    &\qquad=\frac{1}{r+3} \sum_{i=1}^{r+3} \langle \rey{\Scal_{r+2}}(A
    \otimes \one^{\otimes r}), g_i\rangle \langle \one, f_i\rangle\\
    &\qquad\geq 0,
  \end{split}
\]
as we wanted.

To see that~$\kpcone_r(V) \subseteq \cop(V)$, take any~$A \in \kpcone_r(V)$.
For every~$f \in L^2(V)$ such that~$f \geq 0$ and~$\langle \one, f\rangle > 0$
we have
\[
  \begin{split}
    0&\leq \langle \rey{\Scal_{r+2}}(A \otimes \one^{\otimes r}),
    f^{\otimes(r+2)}\rangle\\
    &=\langle A \otimes \one^{\otimes r}, \rey{\Scal_{r+2}}(f \otimes f
    \otimes f^{\otimes r})\rangle\\
    &=\langle A \otimes \one^{\otimes r}, f \otimes f \otimes f^{\otimes
      r}\rangle\\
    &=\langle A, f\otimes f\rangle \langle \one, f\rangle^r,
  \end{split}
\]
whence~$\langle Af, f\rangle = \langle A, f \otimes f\rangle \geq 0$, and
so~$A \in \cop(V)$.
\end{proof}

The cones~$\kpcone_r(V)$ are thus a hierarchy of inner approximations
of~$\cop(V)$.  We will see in~\S\ref{sec:convergence} some sufficient conditions
for this hierarchy to converge, in the sense that the closure of the union of
the cones~$\kpcone_r(V)$ is~$\cop(V)$.

The dual cones~$\kpcone_r^*(V)$ provide a hierarchy of outer approximations
of~$\cp(V)$, namely
\[
  \kpcone_1^*(V) \supseteq \kpcone_2^*(V) \supseteq \cdots \supseteq
  \cp(V).
\]
Given a graph~$G = (V, E)$, set
\[
  \begin{optprob}
    \gamma_r(G)=\sup&\langle A, J\rangle\\
    &\int_V A(x, x)\, d\omega(x) = 1,\\
    &A(x, y) = 0\quad\text{if~$xy \in E$,}\\
    &\onerow{\text{$A \in \kpcone_r^*(V)$ is continuous and positive
        semidefinite.}}
  \end{optprob}
\]
This gives a hierarchy of bounds for the measurable independence number, namely
\[
  \gamma_1(G) \geq \gamma_2(G) \geq \cdots \geq \alpha_\omega(G).
\]
In~\S\ref{sec:convergence} we will prove that this hierarchy converges to the
measurable independence number for some classes of graphs.  For finite graphs,
this hierarchy was proposed (without the positive-semidefiniteness constraint)
by De Klerk and Pasechnik~\cite{KlerkP2002a}.

%=====================================================================

\subsection{Comparison to the Lasserre hierarchy and the $k$-point bound}

Both the Lasserre hierarchy and the $k$-point bound of~\S\ref{sec:lass-k-point}
are eventually at least as good as the completely positive hierarchy, hence if
the latter converges, so do the other two.

\begin{theorem}%
  \label{thm:lass-cop-comparison}
  If~$G = (V, E)$ is a topological graph where~$V$ is a compact Hausdorff
  space and if~$\omega$ is a Radon measure on~$V$, then~$\lass{r+2}(G) \leq
  \Delta_{r+3}(G) \leq \gamma_r(G)$ for every~$r \geq 1$.
\end{theorem}

\begin{proof}
Fix an integer~$r \geq 1$.  Given a feasible solution of~$\Delta_{r+3}(G)$, we
will construct a feasible solution of~$\gamma_r(G)$ with at least the same
objective value, thus proving that~$\Delta_{r+3}(G) \leq \gamma_r(G)$.
Since~$\lass{k}(G) \leq \Delta_{k+1}(G)$ for all~$k \geq 1$, as shown
in~\S\ref{sec:lass-k-point}, the other inequality in the statement will follow as
well.

Thus let~$\nu \in C(\sub{V}{r+3})$ be a feasible solution of~$\Delta_{r+3}(G)$
with positive objective value.  Consider the function~$F\colon V^{r+2} \to \R$
such that
\[
  F(x, y, v) = \nu(\flatten{v} \cup \{x, y\}) = (M_{\flatten{v}}\nu)(\{x\},
  \{y\}),
\]
where~$x$, $y \in V$ and~$v \in V^r$.  Note that~$F$ is continuous and invariant
under~$\Scal_{r+2}$, so~$\Tcal_r^* F$ is continuous and given by
\[
  \begin{split}
    (\Tcal_r^* F)(x, y) &= \int_{V^r} F(x, y, v)\, d\omega(v)\\
    &= \int_{V^r} \nu(\flatten{v} \cup \{x, y\})\, d\omega(v)\\
    &= \int_{V^r} (M_{\flatten{v}}\nu)(\{x\}, \{y\})\, d\omega(v).
  \end{split}
\]

Since~$M_{\flatten{v}}\nu$ is positive semidefinite for every~$v \in V^r$, it is
clear that~$\Tcal_r^* F$ is positive semidefinite.  We claim
that~$\Tcal_r^* F \in \kpcone_r^*(V)$.  Indeed, if~$\flatten{v} \cup \{x,y\}$ is
not independent, then~$F(x, y, v) = 0$; otherwise we
have~$F(x, y, v) = (M_{\flatten{v}\cup\{y\}}\nu)(\{x\}, \{x\}) \geq 0$
since~$M_Q\nu$ is positive semidefinite for every independent set~$Q$
with~$|Q| \leq r+1$.  So~$F \geq 0$ and for every kernel~$Z \in \kpcone_r(V)$ we
have~$\langle Z, \Tcal_r^* F\rangle = \langle \Tcal_r Z, F\rangle \geq 0$,
whence~$\Tcal_r^*F \in \kpcone_r^*(V)$.

We clearly have~$(\Tcal_r^*F)(x, y) = 0$ if~$xy \in E$.
Let~$\tau = \int_V (\Tcal_r^*F)(x, x)\, d\omega(x)$; we will see soon
that~$\tau > 0$.  Then~$A = \tau^{-1} \Tcal_r^* F$ will be a feasible solution
of~$\gamma_r(G)$; our goal is then to estimate its objective value.

To simplify notation, we define~$V^0 = \{\emptyset\}$ and
set~$\flatten{\emptyset} = \emptyset$; we also denote by~$\omega$ the counting
measure on~$V^0$.  For an integer~$0\leq t \leq r+3$ write
\[
  \Phi_t = \int_{V^t} \nu(\flatten{v})\, d\omega(v).
\]
We claim that the matrix
\begin{equation}%
  \label{eq:cont-norm-matrix}
  \begin{pmatrix}
    \Phi_t&\Phi_{t+1}\\
    \Phi_{t+1}&\Phi_{t+2}
  \end{pmatrix}
\end{equation}
is positive semidefinite for all~$0 \leq t \leq r+1$.

Indeed, fix~$0 \leq t \leq r+1$ and let~$B\colon \sub{V}{1}^2 \to \R$ be such
that
\[
  B(S, T) = \int_{V^t} \nu(\flatten{v} \cup S \cup T)\, d\omega(v) =
  \int_{V^t} (M_{\flatten{v}}\nu)(S, T)\, d\omega(v).
\]
It is at once clear that~$B$ is a positive-semidefinite kernel and
that~$B(S, T)$ depends only on~$S \cup T$.  Moreover,
\begin{flalign*}
  B(\emptyset, \emptyset)&=\int_{V^t} \nu(\flatten{v})\, d\omega(v) =
  \Phi_t,\\
  \int_V B(\emptyset, \{x\})\, d\omega(x) &= \int_V \int_{V^t}
  \nu(\flatten{v} \cup \{x\})\, d\omega(v)d\omega(x) = \Phi_{t+1},\text{
    and}\\
  \int_V\int_V B(\{x\}, \{y\})\, d\omega(y)d\omega(x) &= \int_V \int_V
  \int_{V^t} \nu(\flatten{v} \cup \{x,y\})\, d\omega(v) d\omega(y)
  d\omega(x)\\
  &=\Phi_{t+2},
\end{flalign*}
and so the matrix in~\eqref{eq:cont-norm-matrix} is positive semidefinite.  It
follows that, since~$\Phi_0 = 1$ and since~$\nu$ has objective
value~$\Phi_1 > 0$, we need to have~$\Phi_2 > 0$ as well.  Repeating the
argument, we see that~$\Phi_t > 0$ for all~$t$.

Hence for every fixed~$t$ we have~$\Phi_t \Phi_{t+2} - \Phi_{t+1}^2 \geq 0$,
whence~$\Phi_{t+2}\Phi_{t+1}^{-1} \geq \Phi_{t+1} \Phi_t^{-1}$.  Apply this inequality
repeatedly to get~$\Phi_{r+2}\Phi_{r+1}^{-1} \geq \Phi_1 \Phi_0^{-1} = \Phi_1$.
Now~$\tau = \Phi_{r+1}$, and so~$\tau > 0$.
Moreover,~$\Phi_{r+2} = \langle \Tcal_r^*F, J\rangle$,
hence~$\langle A, J\rangle \geq \Phi_1 = \int_V \nu(\{x\})\, d\omega(x)$, as we
wanted.
\end{proof}

%%%%%%%%%%%%%%%%%%%%%%%%%%%%%%%%%%%%%%%%%%%%%%%%%%%%%%%%%%%%%%%%%%%%%%

\section{Convergence}%
\label{sec:convergence}

The Lasserre hierarchy converges for finite graphs, that is, if~$G$ is any
finite graph then~$\lass{k}(G) = \alpha(G)$ for all~$k \geq \alpha(G)$.  De
Klerk and Pasechnik~\cite{KlerkP2002a} showed that the completely positive
hierarchy~$\gamma_r$ also converges for finite graphs,
namely~$\gamma_r(G) \to \alpha(G)$ as~$r \to \infty$.
Since~$\lass{k}(G) = \lass{k'}(G)$ for all~$k$, $k' \geq \alpha(G)$, the
convergence of the Lasserre hierarchy then follows from
Theorem~\ref{thm:lass-cop-comparison}.  Actually, we also get in this way
convergence for the $k$-point bound.

De Laat and Vallentin~\cite{LaatV2015} showed that an extension of the Lasserre
hierarchy converges also for topological packing graphs.  Kuryatnikova and
Vera~\cite{Kuryatnikova2019} extended the completely positive
hierarchy~$\gamma_r$ from finite graphs to topological packing graphs and showed
that it converges.

In this section we will discuss sufficient conditions under which the completely
positive hierarchy~$\gamma_r$ converges for locally independent graphs and,
using Theorem~\ref{thm:lass-cop-comparison}, we will obtain convergence results
for the Lasserre hierarchy and the $k$-point bound as consequences.  The first
step is to better understand the copositive cone~$\cop(V)$ and its inner
approximations~$\kpcone_r(V)$.

%=====================================================================

\subsection{The copositive cone and its inner approximations}%
\label{sec:cop-inner-approx}

Let~$X$ be a real vector space and let~$S \subseteq X$.  We say that~$x \in S$
is in the \defi{algebraic interior} of~$S$ if~$0$ is in the interior of~$\{\,
\lambda \in \R : x + \lambda y \in S\,\}$ for every~$y \in X$.  If~$X$ is a
topological vector space, then the interior of~$S$ is a subset of its algebraic
interior.

Let~$V$ be a compact Hausdorff space equipped with a Radon measure~$\omega$.
Write
\[
  \copc(V) = \{\, A \in \cop(V) : \text{$A$ is continuous}\,\}.
\]

In~\S\ref{sec:func-analysis} we mentioned the observation of
Bochner~\cite{Bochner1941} that if~$\omega$ is positive on open sets, then a
continuous kernel is positive semidefinite if and only if each of its finite
principal submatrices is positive semidefinite.  The same holds for copositive
kernels~\cite[Theorem~4.7]{DeCorteOV2022}:

\begin{theorem}%
  \label{thm:cop-bochner}
  Let~$V$ be a compact Hausdorff space equipped with a Radon measure~$\omega$
  that is positive on open sets.  A kernel~$A \in \csym(V)$ is copositive if and
  only if for every finite set~$U \subseteq V$ the
  matrix~$\bigl(A(x, y)\bigr)_{x,y \in U}$ is copositive.
\end{theorem}

As a subset of the vector space~$\csym(V)$, the cone~$\copc(V)$ has nonempty
algebraic interior: the constant-one kernel~$J$ is, for instance, in its
algebraic interior.  If we equip~$\csym(V)$ with the supremum norm, then~$J$
is also in the interior of~$\copc(V)$.  In
contrast,~$\cop(V) \subseteq \lsym(V)$ has, in general, empty algebraic
interior and hence empty interior, as we will see later.

Kuryatnikova and Vera~\cite{Kuryatnikova2019} showed the following result,
extending a theorem of De Klerk and Pasechnik~\cite{KlerkP2007} (for
completeness, a proof can be found in Appendix~\ref{apx:cop-hierarchy}):

\begin{theorem}%
  \label{thm:olga-juan}
  If~$V$ is a compact Hausdorff space equipped with a Radon measure that is
  positive on open sets, then every kernel in the algebraic interior
  of $\copc(V) \subseteq \csym(V)$ belongs to a cone~$\kpcone_r(V)$ for
  some~$r \geq 1$.
\end{theorem}

Therefore the cones~$\kpcone_r(V)$ somehow capture~$\copc(V)$, since every
kernel in $\copc(V)$ can be approximated in supremum norm arbitrarily well by
kernels in the algebraic interior of~$\copc(V)$.  We need an analogous result
for~$\cp(V)$ and the outer approximations~$\kpcone_r^*(V)$.

\begin{theorem}%
  \label{thm:cp-convergence}
  Let~$\Gamma$ be a compact group and~$V$ be a homogeneous $\Gamma$-space
  equipped with the pushforward of the Haar measure on~$\Gamma$.
  If~$A \in \lsym(V)$ is such that $\langle A, Z\rangle \geq 0$ for
  all~$Z \in \bigcup_{r\geq 1} \kpcone_r(V)$, then~$A$ is completely positive.
\end{theorem}

Why is the theorem stated only for kernels on homogeneous spaces?  In the proof
we use that~$\copc(V)$ is dense in~$\cop(V)$, and though this may well be true
in general, the proof provided in Lemma~\ref{lem:copc-dense} below uses the
group action to approximate a copositive kernel by a continuous copositive
kernel.  The analogous result for positive-semidefinite (instead of copositive)
kernels, namely that a positive-semidefinite kernel in~$\lsym(V)$ can be
approximated arbitrarily well by continuous positive-semidefinite kernels,
follows from the spectral theorem; it is the lack of such a characterization of
copositive kernels that explains the less general result above.

\begin{lemma}%
  \label{lem:copc-dense}
  If~$\Gamma$ is a compact group and if~$V$ is a homogeneous $\Gamma$-space
  equipped with the pushforward~$\omega$ of the Haar measure~$\mu$ on~$\Gamma$,
  then~$\copc(V)$ is dense in~$\cop(V)$ in the~$L^2$-norm topology.
\end{lemma}

\begin{proof}
We first prove the statement for~$V = \Gamma$.  Let~$\Bcal$ be a neighborhood
base of~$1 \in \Gamma$ consisting of compact symmetric
sets~\cite[Proposition~2.1]{Folland1995}.
Then~$\psi_U = \mu(U)^{-2} \chi_U \otimes \chi_U$ for~$U \in \Bcal$ is an
approximate identity for the direct product~$\Gamma \times \Gamma$ (see
Folland~\cite[\S2.5]{Folland1995}).  Note
that~$\psi_U(\sigma^{-1}, \tau^{-1}) = \psi_U(\sigma, \tau)$ for all~$\sigma$,
$\tau \in \Gamma$, since the sets in~$\Bcal$ are symmetric.

Take~$A \in \cop(\Gamma)$.  For every~$U \in \Bcal$, the convolution
\[
  \begin{split}
    (\psi_U * A)(\sigma, \tau) &= \int_{\Gamma \times \Gamma}
    \psi_U(\alpha, \beta) A(\alpha^{-1}\sigma, \beta^{-1}\tau)\,
    d\mu(\alpha, \beta)\\
    &=\int_{\Gamma \times \Gamma} \psi_U(\sigma\alpha, \tau\beta)
    A(\alpha^{-1}, \beta^{-1})\, d\mu(\alpha, \beta)\\
    &=\int_{\Gamma \times \Gamma} \psi_U(\alpha^{-1}\sigma^{-1},
    \beta^{-1}\tau^{-1}) A(\alpha^{-1}, \beta^{-1})\, d\mu(\alpha, \beta)\\
    &=\int_{\Gamma \times \Gamma} \psi_U(\alpha\sigma^{-1}, \beta\tau^{-1})
    A(\alpha, \beta)\, d\mu(\alpha, \beta)
  \end{split}
\]
is continuous~\cite[Proposition~2.39(d)]{Folland1995}.
Moreover,~$\|\psi_U * A - A\|_2\to 0$ as~$U \to \{1\}$, implying that for
every~$\epsilon > 0$ there is~$U \in \Bcal$ such that
$\|\psi_U * A - A\|_2 < \epsilon$ (see
Folland~\cite[Proposition~2.42]{Folland1995}).

Finally,~$\psi_U * A$ is copositive for all~$U \in \Bcal$.  Indeed,
given~$f \in L^2(\Gamma)$ with~$f \geq 0$, let~$g \in L^2(\Gamma)$ be given by
\[
  g(\alpha) = \int_\Gamma f(\sigma) \chi_U(\alpha\sigma^{-1})\,
  d\mu(\sigma);
\]
note that~$g \geq 0$.  Since~$A$ is copositive we then have
\[
  \begin{split}
    \langle (\psi_U * A) f, f\rangle &= \mu(U)^{-2} \int_\Gamma\int_\Gamma
    \int_{\Gamma \times \Gamma} A(\alpha, \beta)
    \chi_U(\alpha\sigma^{-1})\chi_U(\beta\tau^{-1})\, d\mu(\alpha, \beta)\\
    &\phantom{= \mu(U)^{-2} \int_\Gamma\int_\Gamma \int_{\Gamma \times
        \Gamma} A(\alpha, \beta)
      \chi_U(\alpha\sigma^{-1})\chi_U(\beta\tau^{-1})}{}\cdot f(\sigma) f(\tau)\,
    d\tau d\sigma\\
    &=\mu(U)^{-2} \int_\Gamma\int_\Gamma A(\alpha, \beta) g(\alpha)
    g(\beta)\, d\mu(\beta) d\mu(\alpha)\\
    &\geq 0,
  \end{split}
\]
as we wanted.

So we see that~$\copc(\Gamma)$ is dense in~$\cop(\Gamma)$.  Now suppose~$V$ is a
homogeneous $\Gamma$-space and let~$A \in \cop(V)$.  Use the trick
from~\S\ref{sec:topo-groups}: fix~$e \in V$ and consider $F \in L^2(\Gamma^2)$
given by~$F(\sigma, \tau) = A(\sigma e, \tau e)$.  Since~$A$ is copositive, so
is~$F$; moreover, for every~$\epsilon > 0$ there is~$U \in \Bcal$
with~$\|\psi_U * F - F\|_2 < \epsilon$.

If~$\sigma_1 e = \sigma_2 e$ and~$\tau_1 e = \tau_2 e$, then
\[
  \begin{split}
    (\psi_U * F)(\sigma_1, \tau_1) &= \int_{\Gamma \times \Gamma}
    \psi_U(\alpha, \beta) A(\alpha^{-1} \sigma_1 e, \beta^{-1} \tau_1 e)\,
    d\mu(\alpha, \beta)\\
    &= \int_{\Gamma \times \Gamma} \psi_U(\alpha, \beta) A(\alpha^{-1}
    \sigma_2 e, \beta^{-1} \tau_2 e)\, d\mu(\alpha, \beta)\\
    &= (\psi_U * F)(\sigma_2, \tau_2),
  \end{split}
\]
hence there is a function~$A_U\colon V^2 \to \R$ such
that~$(\psi_U * F)(\sigma, \tau) = A_U(\sigma e, \tau e)$; since~$\psi_U * F$ is
continuous, so is~$A_U$, and since~$\psi_U * F$ is copositive,~$A_U$ is also
copositive (use Theorem~\ref{thm:cop-bochner}).
Moreover,~$\|A_U - A\|_2 = \|\psi_U * F - F\|_2 < \epsilon$, and we are done.
\end{proof}

\begin{proof}[Proof of Theorem~\ref{thm:cp-convergence}]
If~$Z \in \copc(V)$, then for every~$\lambda \in (0, 1]$ the
kernel~$\lambda J + (1 - \lambda) Z$ is in the algebraic interior of~$\copc(V)$.
The theorem then follows from Theorem~\ref{thm:olga-juan} together with
Lemma~\ref{lem:copc-dense}.
\end{proof}

Let us close this section by showing that~$\cop(V)$ may have empty algebraic
interior.  Let~$V$ be any measure space equipped with a finite measure~$\omega$
such that~$V$ can be partitioned into infinitely many sets~$P_1$, $P_2$, \dots\
of positive measure.  Take any~$A \in \cop(V)$.

For an integer~$n \geq 1$ write
\[
  f(n) = \frac{1}{\omega(P_n)^2} \int_{P_n} \int_{P_n} A(x, y)\, d\omega(y)
  d\omega(x).
\]
Note that, since~$A$ is copositive,~$f(n) \geq 0$ for all~$n$.  If for some~$n$
we have~$f(n) = 0$, then for every~$\lambda > 0$ we have
\[
  \begin{split}
    \int_{P_n} \int_{P_n} (A - \lambda J)(x, y)\, d\omega(y) d\omega(x) &=
    f(n) \omega(P_n)^2 - \lambda \omega(P_n)^2\\
    &=(f(n) - \lambda) \omega(P_n)^2\\
    &=-\lambda \omega(P_n)^2 < 0,
  \end{split}
\]
and so~$A - \lambda J$ is not copositive and hence~$A$ is not in the algebraic
interior of~$\cop(V)$.  So assume~$f(n) > 0$ for all~$n$.

Use Cauchy-Schwarz to get
\[
  \biggl(\int_{P_n}\int_{P_n} A(x, y)\, d\omega(y) d\omega(x)\biggr)^2 \leq
  \omega(P_n)^2 \int_{P_n}\int_{P_n} A(x, y)^2\, d\omega(y) d\omega(x),
\]
whence~$\sum_{n \geq 1} f(n)^2 \omega(P_n)^2 < \infty$.

Now let~$g(n)$ be such that~$g(n) \to \infty$ and~$\sum_{n \geq 1} f(n)^2
g(n)^2 \omega(P_n)^2 < \infty$.  Set
\[
  K = \sum_{n \geq 1} f(n) g(n) \chi_{P_n} \otimes \chi_{P_n}.
\]
Note that~$\|K\|_2^2 = \sum_{n \geq 1} f(n)^2 g(n)^2 \omega(P_n)^2 <
\infty$, so~$K \in \lsym(V)$.  Moreover, for every~$\lambda > 0$ we have
\[
  \begin{split}
    \int_{P_n}\int_{P_n} (A - \lambda K)(x, y)\, d\omega(y) d\omega(x) &=
    f(n) \omega(P_n)^2 - \lambda f(n) g(n) \omega(P_n)^2\\
    &=f(n) \omega(P_n)^2 (1 - \lambda g(n)),
  \end{split}
\]
and since~$g(n) \to \infty$ and~$f(n) > 0$ for all~$n$ we see that for
every~$\lambda > 0$ the kernel $A - \lambda K$ is not copositive, and hence~$A$
is not in the algebraic interior of~$\cop(V)$.

%=====================================================================

\subsection{Convergence for distance graphs}%
\label{sec:dist-convergence}

Given a locally independent graph~$G = (V, E)$ and a measure~$\omega$ on~$V$,
our goal is to show that~$\gamma_r(G)$ converges to~$\alpha_\omega(G)$.  Then,
using Theorem~\ref{thm:lass-cop-comparison}, we get convergence of both the
Lasserre hierarchy and the $k$-point bound.

Our strategy is as follows.  For~$r \geq 1$, let~$A_r$ be a feasible solution
of~$\gamma_r(G)$.  Note that the feasible region of~$\gamma_r(G)$ is contained
in the feasible region of~$\gamma_1(G)$.  So, if the feasible region
of~$\gamma_1(G)$ is contained in a compact set, then the sequence~$(A_r)$ will
have a converging subsequence, say with limit~$A$.  The goal is then to show
that~$A$ is a feasible solution of the completely positive
formulation~$\vartheta(G, \cp(V))$, so we can finish by using the result of
DeCorte, Oliveira, and Vallentin~\cite{DeCorteOV2022}.

Summarizing, we need to
\begin{enumerate}
\item show that the feasible region of~$\gamma_1(G)$ is contained in some
  compact set in a convenient topology and

\item show that the limit solution~$A$ is feasible for~$\vartheta(G, \cp(V))$.
\end{enumerate}
There is an interplay between~(1) and~(2), since the topology influences what we
know about the limit~$A$, and so if we choose a bad topology,~$A$ will not be
feasible.  As for~(2), the main issue seems to be the
constraints~``$A(x, y) = 0$ for~$xy \in E$''; depending on the graph,~$A$ may
not satisfy these constraints even though each~$A_r$ does.

For distance graphs on the sphere (and related spaces; see below), this strategy
works.

\begin{theorem}%
  \label{thm:sphere-convergence}
  Let~$\omega$ be the surface measure on~$S^{n-1}$.  If~$n \geq 3$ and
  if~$D \subseteq (-1, 1)$ is closed,
  then~$\lim_{r \to \infty} \gamma_r(G) = \alpha_\omega(G)$
  where~$G = G(S^{n-1}, D)$.
\end{theorem}

So the~$\gamma_r$ hierarchy converges for Witsenhausen's
graph~$G = G(S^{n-1}, \{0\})$, as long as~$n \geq 3$.  For~$n = 2$, the theorem
does not provide any information, even though in this
case~$\gamma_1(G) = \alpha_\omega(G)$
(cf.~Oliveira~\cite[\S3.5a]{Oliveira2009}).

The proof requires some facts about harmonic analysis on the sphere; see
Andrews, Askey, and Roy~\cite{AndrewsAR1999} for background.  Let~$\omega$ be
the surface measure on~$S^{n-1}$ and write~$\omega_n = \omega(S^{n-1})$.

For~$k \geq 0$, let~$H_k^n \subseteq C(S^{n-1})$ be the space of $n$-variable
spherical harmonics of degree~$k$; denote the dimension of~$H_k^n$ by~$h_k^n$
and let~$S_{k,1}^n$, \dots,~$S_{k,h_k^n}^n$ be an orthonormal basis of~$H_k^n$.
The set of all spherical harmonics~$S_{k,i}^n$ for all~$k$ and~$i$ forms a
complete orthonormal system of~$L^2(S^{n-1})$.

Let~$P_k^n$ be the Jacobi polynomial of degree~$k$ and
parameters~$\alpha = \beta = (n-3)/2$ normalized so~$P_k^n(1) = 1$.  The
addition formula~\cite[Theorem~9.6.3]{AndrewsAR1999} states that
\[
  P_k^n(x \cdot y) = \frac{\omega_n}{h_k^n} \sum_{i=1}^{h_k^n} S_{k,i}^n(x)
  S_{k,i}^n(y)
\]
for all~$x$, $y \in S^{n-1}$.

Write~$E_k^n(x, y) = P_k^n(x \cdot y)$ and let~$\orto(n)$ be the orthogonal group
on~$\R^n$.  The kernels~$E_k^n$ are $\orto(n)$-invariant and form a complete
orthogonal system of the space $L^2((S^{n-1})^2)^{\orto(n)}$ of
$\orto(n)$-invariant kernels.  We have~$\langle E_k^n, E_l^n\rangle = 0$
if~$k \neq l$ and $\langle E_k^n, E_k^n\rangle = \omega_n^2 / h_k^n$.  So
if~$A = \sum_{k=0}^\infty f(k) E_k^n$ and~$B = \sum_{k=0}^\infty g(k) E_k^n$,
then
\begin{equation}%
  \label{eq:schoenberg-product}
  \langle A, B\rangle = \sum_{k=0}^\infty f(k) g(k) \omega_n^2/h_k^n.
\end{equation}

From the addition formula, each~$E_k^n$ is positive semidefinite,
hence~$\sum_{k=0}^\infty f(k) E_k^n$ is positive semidefinite if and only
if~$f(k) \geq 0$ for all~$k$.  Schoenberg's theorem~\cite{Schoenberg1942} states
that if a kernel~$K\colon (S^{n-1})^2 \to \R$ is continuous,
$\orto(n)$-invariant, and positive semidefinite, then there is a nonnegative
sequence~$f\colon \N \to \R$ with~$\sum_{k=0}^\infty f(k) < \infty$ such that
\begin{equation}%
  \label{eq:schoenberg}
  K(x, y) = \sum_{k=0}^\infty f(k) E_k^n(x, y)
  = \sum_{k=0}^\infty f(k) P_k^n(x \cdot y)
\end{equation}
with absolute and uniform convergence over~$S^{n-1} \times S^{n-1}$.

Conversely, if~$f\colon \N \to \R$ is a nonnegative sequence such
that~$\sum_{k=0}^\infty f(k) < \infty$, then the series in~\eqref{eq:schoenberg}
converges absolutely and uniformly over~$S^{n-1} \times S^{n-1}$ and the
kernel~$K$ it defines is continuous, $\orto(n)$-invariant, and positive
semidefinite.

\begin{proof}[Proof of Theorem~\ref{thm:sphere-convergence}]
We know from~\S\ref{sec:cop-hierarchy} that
$\lim_{r \to \infty} \gamma_r(G) \geq \vartheta(G, \cp(S^{n-1}))$; let us show
the reverse inequality, thus establishing the result by using Theorem~5.1 from
DeCorte, Oliveira, and Vallentin~\cite{DeCorteOV2022}.

The cone~$\kpcone_r(S^{n-1})$ is invariant under~$\orto(n) \subseteq \aut(G)$,
so if~$A$ is a feasible solution of~$\gamma_r(G)$, then so is its
symmetrization~$\rey{\orto(n)}A$.  It follows that we may restrict ourselves
in~$\gamma_r(G)$ to $\orto(n)$-invariant kernels.

For~$r \geq 1$, let~$A_r$ be any $\orto(n)$-invariant feasible solution
of~$\gamma_r(G)$ such
that $\langle A_r, J\rangle \geq \alpha_\omega(G) / 2 > 0$.  For
each~$r \geq 1$, use Schoenberg's theorem to get a nonnegative
sequence~$f_r\colon \N \to \R$ such that~$A_r = \sum_{k=0}^\infty f_r(k) E_k^n$.
Note that
\begin{equation}%
  \label{eq:Ar-trace}
  \omega_n \sum_{k=0}^\infty f_r(k) = \int_{S^{n-1}} A_r(x, x)\, d\omega(x)
  = 1
\end{equation}
for every~$r \geq 1$.

Let~$\ell^1$ be the space of all sequences~$f\colon \N \to \R$ such
that~$\|f\|_1 = \sum_{k=0}^\infty |f(k)| < \infty$ and let~$c_0$ be the space of
all sequences~$f\colon \N \to \R$ that vanish at infinity.  Under the supremum
norm,~$c_0$ is a Banach space and its dual is~$\ell^1$; the duality pairing
is~$(f, g) = \sum_{k=0}^\infty f(k) g(k)$.

Let~$\Bcal = \{\, f \in \ell^1 : \|f\|_1 \leq 1\,\}$; from~\eqref{eq:Ar-trace}
we know that~$\omega_n f_r \in \Bcal$ for all~$r$.  Alaoglu's
theorem~\cite[Theorem~5.18]{Folland1999} says that~$\Bcal$ is weak* compact.
Since the weak* topology on~$\Bcal$ is also metrizable~\cite[p.~171,
Exercise~50]{Folland1999}, the sequence~$(f_r)$ has a converging subsequence;
assume that~$(f_r)$ itself converges to~$f$.  This means that, if~$g \in c_0$,
then $\lim_{r \to \infty} (f_r, g) = (f, g)$.

It follows immediately that~$f$ is nonnegative and, using~\eqref{eq:Ar-trace},
that~$\omega_n\|f\|_1 \leq 1$.  So by Schoenberg's theorem the
kernel~$A = \sum_{k=0}^\infty f(k) E_k^n$ is continuous, $\orto(n)$-invariant,
and positive semidefinite.  Moreover, if~$B \in L^2((S^{n-1})^2)^{\orto(n)}$,
then~$B = \sum_{k=0}^\infty g(k) E_k^n$ for some sequence~$g$, and
since~$k \mapsto g(k) \omega_n^2 / h_k^n$ vanishes at infinity we have
from~\eqref{eq:schoenberg-product} that
\[
  \lim_{r \to \infty} \langle A_r, B\rangle = \langle A, B\rangle.
\]

We then have
that~$\langle A, J\rangle = \lim_{r\to\infty} \langle A_r, J\rangle > 0$, and
so~$A$ is nonzero and
hence~$0 < \tau = \int_V A(x, x)\, d\omega(x) = \omega_n\|f\|_1 \leq 1$.  We
show that~$A$ is completely positive and that~$A(x, y) = 0$
whenever~$x \cdot y \in D$, thus showing that~$\tau^{-1} A$ is a feasible
solution of~$\vartheta(G, \cp(S^{n-1}))$ and hence
that~$\vartheta(G, \cp(S^{n-1})) \geq \lim_{r\to\infty} \gamma_r(G)$, finishing
the proof.

If~$Z \in \kpcone_s(S^{n-1})$ for some~$s$, then
\[
  \langle A, Z\rangle = \langle A, \rey{\orto(n)}Z\rangle =
  \lim_{r\to\infty} \langle A_r, \rey{\orto(n)}Z\rangle \geq 0,
\]
and using Theorem~\ref{thm:cp-convergence} we see that~$A$ is completely
positive.  Next, since~$n \geq 3$ the asymptotic formula for the Jacobi
polynomials~\cite[Theorem~8.21.8]{Szego1975} implies that~$k \mapsto P_k^n(t)$
vanishes at infinity for every~$t \in D \subseteq (-1, 1)$.  So, if~$x$,
$y \in S^{n-1}$ are such that~$x \cdot y \in D$, then
\[
  A(x, y) = \sum_{k=0}^\infty f(k) P_k^n(x \cdot y) =
  \lim_{r\to\infty}\sum_{k=0}^\infty f_r(k) P_k^n(x\cdot y) = 0,
\]
and~$A(x, y) = 0$ whenever~$x\cdot y \in D$, as we wanted.
\end{proof}

The reason we had to require~$n \geq 3$ in Theorem~\ref{thm:sphere-convergence}
is to get the sequence $k \mapsto P_k^n(t)$ to vanish at infinity for~$t \in D$
so the limit~$A$ would satisfy the edge constraints.  The same trick works also
for manifolds similar to the sphere, namely the continuous, compact, two-point
homogeneous spaces~\cite{Wang1952}: the sphere, the real, complex, and
quaternionic projective spaces, and the octonionic projective plane.  For these
spaces, a theorem like Schoenberg's theorem~\cite[Theorem~3.1]{OliveiraV2013}
gives a characterization of continuous, invariant, and positive-semidefinite
kernels in terms of an expansion like~\eqref{eq:schoenberg} involving Jacobi
polynomials.  As long as the real dimension of the space is at least~2, we
get sequences vanishing at infinity as we like.  So for these spaces both
Theorem~\ref{thm:sphere-convergence} and its proof can be easily adapted.

%=====================================================================

\subsection{Convergence on homogeneous spaces}

Let~$X$ be a topological space and~$\mu$ be a Borel measure on~$X$.  We say that
a measurable set~$U \subseteq X$ is \defi{$\mu$-thick} if for every open
set~$A \subseteq X$ such that~$A \cap U \neq \emptyset$ we
have~$\mu(A \cap U) > 0$.

\begin{theorem}%
  \label{thm:thick-convergence}
  Let~$G = (V, E)$ be a locally independent graph and
  let~$\Gamma \subseteq \aut(G)$ be a compact group.  If~$V$ is a homogeneous
  $\Gamma$-space, if~$\omega$ is the pushforward to~$V$ of the Haar measure
  on~$\Gamma$, and if~$E$ is $\omega$-thick, then
  $\lim_{r\to\infty} \gamma_r(G) = \vartheta(G, \cp(V))$.
\end{theorem}

An example of a graph satisfying the hypotheses of the theorem
is~$G(S^{n-1}, [a, b])$ where~$-1 \leq a < b < 1$.  Another example, not covered
by Theorem~\ref{thm:sphere-convergence}, is the graph on the torus~$\R^n / \Z^n$
where two points are adjacent if their geodesic distance lies in~$[a, b]$.  In
contrast, Witsenhausen's graph~$G(S^{n-1}, \{0\})$ does not satisfy the
hypotheses.

Note that the theorem does not state that~$\gamma_r(G) \to \alpha_\omega(G)$.
To prove this we need to know that~$\vartheta(G, \cp(V)) = \alpha_\omega(G)$,
and for this we need a few extra technical
assumptions~\cite[Theorem~5.1]{DeCorteOV2022} on~$\Gamma$ and~$G$.

For the proof we will need the following lemma.

\begin{lemma}%
  \label{lem:trace-class-cont}
  Let~$\Gamma$ be a compact group,~$V$ be a homogeneous $\Gamma$-space,
  and~$\omega$ be the pushforward to~$V$ of the Haar measure on~$\Gamma$.
  If~$A \in \lsym(V)$ is trace class, then~$\rey{\Gamma}A$ is continuous.
\end{lemma}

\begin{proof}
Consider the spectral decomposition of~$A$, namely
\[
  A = \sum_{k=1}^\infty \lambda_k f_k \otimes f_k
\]
with~$L^2$ convergence, where~$f_k \in L^2(V)$ are pairwise orthogonal
and~$\|f_k\|_2 = 1$ for all~$k$.  Since~$A$ is trace class,~$\sum_{k=1}^\infty
|\lambda_k| < \infty$.

With~$E_k = \rey{\Gamma}(f_k \otimes f_k)$ we have~$\rey{\Gamma}A =
\sum_{k=1}^\infty \lambda_k E_k$ with~$L^2$ convergence. From
Lemma~\ref{lem:tensor-reynolds}, each~$E_k$ is continuous.  Since~$\Gamma$ acts
transitively, we also have~$E_k(x, x) = \langle f_k, f_k\rangle = 1$ for
all~$k$.

Each kernel~$E_k$ is continuous and positive semidefinite, so given~$x$,
$y \in V$ the matrix
\[
  \begin{pmatrix}
    E_k(x, x)&E_k(x, y)\\
    E_k(x, y)&E_k(y, y)
  \end{pmatrix}
\]
is positive semidefinite, whence~$|E_k(x, y)| \leq 1$.  Given~$\epsilon > 0$,
let~$N$ be such that~$\sum_{k=N}^\infty |\lambda_k| < \epsilon$.  Then for
all~$x$, $y \in V$ we have
\[
  \biggl|\sum_{k=N}^\infty \lambda_k E_k(x, y)\biggr| \leq \sum_{k=N}^\infty
  |\lambda_k| |E_k(x, y)| \leq \sum_{k=N}^\infty |\lambda_k| < \epsilon.
\]
So the series~$\sum_{k=1}^\infty \lambda_k E_k(x, y)$ converges absolutely and
uniformly over $V \times V$, and~$A$ is continuous.
\end{proof}

\begin{proof}[Proof of Theorem~\ref{thm:thick-convergence}]
We know from~\S\ref{sec:cop-hierarchy} that
$\lim_{r\to\infty} \gamma_r(G) \geq \vartheta(G, \cp(V))$; let us show the
reverse inequality.

The cone~$\kpcone_r(V)$ is invariant under~$\Gamma$, so if~$A$ is a feasible
solution of~$\gamma_r(G)$, then so is its symmetrization~$\rey{\Gamma}A$.  It
follows that we may restrict ourselves in~$\gamma_r(G)$ to $\Gamma$-invariant
kernels.

For~$r \geq 1$, let~$A_r$ be any $\Gamma$-invariant feasible solution
of~$\gamma_r(G)$ such that~$\langle A, J\rangle \geq \alpha_\omega(G) / 2 > 0$.
Since~$A_r$ is
$\Gamma$-invariant,~$A_r(x, x) = \int_V A_r(y, y)\, d\omega(y) = 1$ for
all~$x$.  Since~$A_r$ is continuous and positive semidefinite, for
all~$x$, $y \in V$ the matrix
\[
  \begin{pmatrix}
    A_r(x, x)&A_r(x, y)\\
    A_r(x, y)&A_r(y, y)
  \end{pmatrix}
\]
is positive semidefinite, whence~$|A_r(x, y)| \leq 1$.  So for~$r \geq 1$ we
have~$\|A_r\|_2 \leq 1$,
thus~$A_r \in \Bcal = \{\, F \in \lsym(V) : \|F\|_2 \leq 1\,\}$ for
all~$r \geq 1$.

Alaoglu's theorem~\cite[Theorem~5.18]{Folland1999} says that~$\Bcal$ is compact
in the weak topology of~$\lsym(V)$.  The weak topology on~$\Bcal$ is moreover
metrizable~\cite[p.~171, Exercise~50]{Folland1999}, so the sequence~$(A_r)$ has
a weakly converging subsequence; assume that the sequence itself converges
to~$A$.

This means that for all~$B \in \lsym(V)$ we have
\[
  \lim_{r\to\infty} \langle A_r, B\rangle = \langle A, B\rangle.
\]
In
particular,~$\langle A, J\rangle = \lim_{r\to\infty} \langle A_r, J\rangle > 0$,
and so~$A \neq 0$.

If~$Z \in \kpcone_s(V)$ for some~$s$,
then~$\langle A, Z\rangle = \lim_{r\to\infty} \langle A_r, Z\rangle \geq 0$, and
so from Theorem~\ref{thm:cp-convergence} it follows that~$A$ is completely
positive.  We claim that~$A$ is trace class and~$\trace A \leq 1$.

Indeed, consider the spectral decomposition
\[
  A = \sum_{k=1}^\infty \lambda_k f_k \otimes f_k
\]
of~$A$, where~$f_k \in L^2(V)$ are pairwise orthogonal and~$\|f_k\|_2 = 1$ for
all~$k$.  Since~$A$ is positive semidefinite (as it is completely positive), we
have~$\lambda_k \geq 0$ for all~$k$.  Since every~$A_r$ is positive semidefinite
as well, for every~$N \geq 1$ we have
\[
  \begin{split}
    \sum_{k=1}^N \lambda_k &= \Bigl\langle A, \sum_{k=1}^N f_k \otimes
    f_k\Bigr\rangle\\
    &=\lim_{r\to\infty} \Bigl\langle A_r, \sum_{k=1}^N f_k \otimes
    f_k\Bigr\rangle\\
    &\leq \lim_{r\to\infty} \trace A_r = 1,
  \end{split}
\]
hence~$\trace A \leq 1$.

From Lemma~\ref{lem:trace-class-cont} we see that~$\rey{\Gamma}A$ is
continuous.  Then, since~$E$ is $\omega$-thick, $(\rey{\Gamma}A)(x, y) = 0$ for
all~$xy \in E$ if and only if~$\langle\rey{\Gamma}A, B\rangle = 0$ for all~$B
\in \lsym(V)$ with support contained in~$E$.  For any such~$B$, since
also~$\Gamma \subseteq \aut(G)$, we have
\[
  \langle\rey{\Gamma}A, B\rangle = \langle A, \rey{\Gamma}B\rangle =
  \lim_{r\to\infty} \langle A_r, \rey{\Gamma}B\rangle = 0,
\]
and so~$(\rey{\Gamma} A)(x, y) = 0$ for all~$xy \in E$.

Finally, since~$A$ is nonzero and positive semidefinite we have $\trace A > 0$.
From Mercer's theorem we see
that~$\int_V (\rey{\Gamma} A)(x, x)\, d\omega(x) = \trace A \leq 1$.
So~$(\trace A)^{-1} (\rey{\Gamma} A)$ is a feasible solution
of~$\vartheta(G, \cp(V))$ and
hence~$\lim_{r\to\infty} \gamma_r(G) \leq \vartheta(G, \cp(V))$, as we wanted.
\end{proof}

%%%%%%%%%%%%%%%%%%%%%%%%%%%%%%%%%%%%%%%%%%%%%%%%%%%%%%%%%%%%%%%%%%%%%%

\section{Computations}%
\label{sec:computations}

We have seen three hierarchies for the independence number of locally
independent graphs and conditions under which they converge, but can we use them
to get better bounds for, say, Witsenhausen's problem?

It turns out that the Lasserre hierarchy and the $k$-point bound are hard to
compute in this case, as we will see in~\S\ref{sec:3-pt-bound}.  This motivates
the introduction of another completely positive hierarchy, stronger than the one
of~\S\ref{sec:cop-hierarchy}, which we use to obtain better bounds for
Witsenhausen's problem (see Tables~\ref{tab:bounds} and~\ref{tab:detailed}).

In this section, we write~$\langle X, Y\rangle = \trace X^\tp Y$ for the trace
inner product between matrices~$X$, $Y \in \R^{n \times n}$.

%=====================================================================

\subsection{The $3$-point bound}%
\label{sec:3-pt-bound}

Our goal is to compute~$\Delta_3(G_n)$ for Witsenhausen's
graph~$G_n = G(S^{n-1}, \{0\})$.  To simplify the discussion, we use an
alternative normalization for~\eqref{opt:k-point}, namely
\[
  \begin{optprob}
    \sup&\int_{S^{n-1}}\int_{S^{n-1}} \nu(\{x, y\})\, d\omega(y)
    d\omega(x)\\[2pt]
    &\int_{S^{n-1}} \nu(\{x\}) = 1,\\[2pt]
    &\nu(S) = 0\quad\text{if~$S \in \sub{V}{3}$ is not independent,}\\
    &\text{$M_\emptyset\nu$ is positive semidefinite,}\\
    &\text{$M_{\{u\}}\nu$ is positive semidefinite for every $u \in S^{n-1}$,}\\
    &\nu \in C(\sub{V}{3}).
  \end{optprob}
\]
With this normalization,~$\Delta_3(G_n)$ is at most the optimal value of the
problem above.  It is not hard to give a direct proof of this assertion; the
relation between different normalizations is developed in the proof of
Theorem~\ref{thm:lass-cop-comparison}.

We simplify the problem further by disregarding the empty set in~$M_Q\nu$.  More
precisely, we change the operator~$M_Q$, where~$|Q| \leq 1$, so that it maps the
function~$\nu$ to the kernel~$K \in C((S^{n-1})^2)$ such
that~$K(x, y) = \nu(Q \cup \{x,y\})$.

When solving the resulting problem we may restrict ourselves to
$\orto(n)$-invariant functions~$\nu$, since we can apply the Reynolds operator
to any feasible solution~$\nu$ and obtain as a result an $\orto(n)$-invariant
feasible solution with the same objective value.  Let now~$e \in S^{n-1}$ be the
north pole (or any other fixed point).  Given any~$u \in S^{n-1}$ there is an
orthogonal transformation~$T$ such that~$Tu = e$, hence if~$\nu$ is an
$\orto(n)$-invariant feasible solution, then
\[
  (M_{\{u\}}\nu)(x, y) = \nu(\{u, x, y\}) = \nu(\{e, Tx, Ty\}) =
  (M_{\{e\}}\nu)(Tx, Ty).
\]
It follows that if~$M_{\{e\}}\nu$ is positive semidefinite, then so
is~$M_{\{u\}}\nu$ for every~$u \in S^{n-1}$.  Note moreover that, since~$\nu$ is
$\orto(n)$-invariant, then~$M_{\{e\}}\nu$ is invariant under the
stabilizer~$\stab{\orto(n)}{e}$ of~$e$.  This allows us to rewrite the problem by
considering the two kernels~$A = M_\emptyset \nu$ and~$K = M_{\{e\}}\nu$:
\begin{equation}
  \label{opt:3-pt-rewrite}
  \begin{optprob}
    \sup&\onerow{\int_{S^{n-1}}\int_{S^{n-1}} A(x, y)\, d\omega(y) d\omega(x)}\\[2pt]
    &\onerow{\int_{S^{n-1}} A(x, x)\, d\omega(x) = 1,}\\[2pt]
    &A(x, y) = K(e, Ty)&\text{for every~$x$, $y \in S^{n-1}$ and~$T \in
      \orto(n)$ with $Tx = e$,}\\
    &K(e, x) = K(x, x)&\text{for every~$x \in S^{n-1}$,}\\
    &K(x, y) = 0&\text{if~$\{e, x, y\}$ is not independent,}\\
    &\onerow{\text{$A \in C((S^{n-1})^2)$ is $\orto(n)$-invariant and
        positive semidefinite,}}\\
    &\onerow{\text{$K \in C((S^{n-1})^2)$ is $\stab{\orto(n)}{e}$-invariant
        and positive semidefinite.}}
  \end{optprob}
\end{equation}

Since~$A$ is $\orto(n)$-invariant, Schoenberg's theorem can be used to
express~$A$ in terms of Jacobi polynomials as in~\eqref{eq:schoenberg}.  The
kernel~$K$ is invariant under the stabilizer subgroup~$\stab{\orto(n)}{e}$; the
parametrization of such a kernel was given by Bachoc and
Vallentin~\cite{BachocV2008} and is as follows.

With~$P_k^n$ being the Jacobi polynomial used in~\eqref{eq:schoenberg},
for~$n \geq 2$ and~$k \geq 0$ consider the polynomial
\[
  Q_k^n(u, v, t) = (1 - u^2)^{k/2} (1 - v^2)^{k/2} P_k^n\biggl(\frac{t - uv}{(1
    - u^2)^{1/2} (1 - v^2)^{1/2}}\biggr)
\]
and let~$Y_{k,d}^n$ be the $(d - k + 1) \times (d - k + 1)$ matrix given by
\begin{equation}%
  \label{eq:Y-matrix}
  (Y_{k,d}^n)_{ij}(u, v, t) = u^i v^j Q_k^{n-1}(u, v, t)
\end{equation}
for~$0 \leq i, j \leq d - k$.

If~$K \in C((S^{n-1})^2)$ is $\stab{\orto(n)}{e}$-invariant, then~$K(x, y)$
depends only on~$e\cdot x$, $e\cdot y$, and~$x\cdot y$.  Bachoc and Vallentin
showed that, for any~$d$ and any choice of positive semidefinite
matrices~$F_k \in \R^{(d-k+1) \times (d-k+1)}$, the kernel
\begin{equation}%
  \label{eq:bachocv-kernel}
  K(x, y) = \sum_{k=0}^d \langle F_k, Y_{k,d}^n(e\cdot x, e\cdot y, x\cdot y)\rangle
\end{equation}
is $\stab{\orto(n)}{e}$-invariant and positive semidefinite.  (Note that~$K$ is
continuous by construction, since it is a polynomial on the three inner products
above.)

We would like to say that, conversely, any continuous, positive-semidefinite,
and $\stab{\orto(n)}{e}$-invariant kernel can be expressed as
in~\eqref{eq:bachocv-kernel} with pointwise uniform convergence, like
Schoenberg's theorem asserts for $\orto(n)$-invariant kernels.  This cannot be
the case with~\eqref{eq:bachocv-kernel} as it is given, since it only captures
kernels coming from polynomials.  Even if we use in~\eqref{eq:bachocv-kernel} an
infinite sum and allow the matrices~$F_k$ to be infinite, we can still not
achieve pointwise convergence.  What can be shown~\cite[Theorem~A.8]{Laat2020}
however is that any continuous, positive-semidefinite, and
$\stab{\orto(n)}{e}$-invariant kernel can be uniformly approximated by kernels
of the form~\eqref{eq:bachocv-kernel} by taking larger and larger degree~$d$.

To solve~\eqref{opt:3-pt-rewrite} we parameterize~$A$ using Schoenberg's theorem
and~$K$ using~\eqref{eq:bachocv-kernel}.  A first problem pops up: even if we
take~$d = \infty$ and use~\eqref{eq:bachocv-kernel} with infinite matrices,
since we do not have pointwise convergence the constraint ``$K(x, y) = 0$
if~$\{e, x, y\}$ is not independent'' cannot be equivalently rewritten in terms
of the expansion~\eqref{eq:bachocv-kernel}, hence it is unclear that the
resulting problem gives an upper bound for~$\alpha_\omega(G_n)$.  The same
happens with other constraints in~\eqref{opt:3-pt-rewrite}.

We could circumvent this problem by requiring
that~$K(x, y) \in [-\epsilon, \epsilon]$ for some fixed~$\epsilon > 0$,
rewriting other constraints in a similar fashion.  Still, using the resulting
problem to get a rigorous upper bound for the measurable independence number
remains difficult.

Indeed, to solve the modified problem~\eqref{opt:3-pt-rewrite} we have to fix
the degrees of the polynomials at some point.  Note that we have to solve the
problem (or a relaxation of it) to optimality to get an upper bound.  To do so
rigorously we have to use polynomials of high degree, and since~$K$ is
parameterized by a 3-variable polynomial, the variable matrices become
prohibitively large.

%=====================================================================

\subsection{Slice-positive functions and another hierarchy}

Let~$V$ be a measure space equipped with a measure~$\omega$.  A
function~$F \in L^2(V^{k+2})$ is \defi{slice-positive} if for almost
all~$v \in V^k$ the kernel $(x, y) \mapsto F(x, y, v)$ is positive
semidefinite.  If~$V$ is a compact Hausdorff space and~$\omega$ is a Radon
measure that is positive on open sets, then a continuous
function~$F\colon V^{k+2} \to \R$ is slice-positive if and only
if~$(x, y) \mapsto F(x, y, v)$ is a positive-semidefinite kernel for
every~$v \in V^k$.

% proof: Say F(., ., v) is not psd for some v.  Use Bochner's obs to get a
% finite submatrix that is not psd.  When v varies a little, each entry of the
% matrix varies only a little, and it remains not psd.  So there is a
% neighborhood of v for which every kernel is not psd.

For an integer~$r \geq 1$, let
\[
  \begin{split}
    \vzcone_r(V) = \{\, A \in \lsym(V) :\ &\rey{\Scal_{r+2}}(A \otimes
    \one^{\otimes r} - F) \geq 0\\
    &\text{for some slice-positive~$F\in L^2(V^{r+2})$}\,\}.
  \end{split}
\]

It is immediate that~$\kpcone_r(V) \subseteq \vzcone_r(V)$ for all~$r$.
Moreover,~$\vzcone_r(V) \subseteq \cop(V)$ for all~$r$.  Indeed,
take~$A \in \vzcone_r(V)$ and~$F \in L^2(V^{r+2})$ with
$\rey{\Scal_{r+2}}(A \otimes \one^{\otimes r} - F) \geq 0$.  Given a
nonnegative~$f\in L^2(V)$ with~$\langle \one, f\rangle > 0$ we have
\[
  \begin{split}
    0 &\leq \langle\rey{\Scal_{r+2}}(A \otimes \one^{\otimes r} - F),
    f^{\otimes(r+2)}\rangle\\
    &= \langle A \otimes \one^{\otimes r},
    \rey{\Scal_{r+2}} f^{\otimes(r+2)}\rangle - \langle F,
    \rey{\Scal_{r+2}} f^{\otimes(r+2)}\rangle\\
    &= \langle A, f \otimes f\rangle \langle \one, f\rangle^r - \langle F,
    f^{\otimes(r+2)}\rangle.
  \end{split}
\]
Now, since~$F$ is slice-positive,
\[
  \begin{split}
    \langle F, f^{\otimes(r+2)}\rangle &= \int_{V^r} \int_V \int_V F(x, y, v)
    f(x) f(y)\, d\omega(y) d\omega(x)\, f(v_1) \cdots f(v_r)\, d\omega(v)\\
    &\geq 0,
  \end{split}
\]
and we see that~$\langle A, f \otimes f\rangle \geq 0$, and so~$A$ is
copositive.

If $A \in \vzcone_r(V)$, that is, if
$\rey{\Scal_{r+2}}(A \otimes \one^{\otimes r} - F) \geq 0$ for some
slice-positive~$F$, then as in the proof of Theorem~\ref{thm:kpcone-hierarchy}
one shows that
$\rey{\Scal_{r+2}}(A \otimes \one^{\otimes r} \otimes \one - F \otimes \one)
\geq 0$, and so~$A \in \vzcone_{r+1}(V)$.  Hence
\[
  \vzcone_1(V) \subseteq \vzcone_2(V) \subseteq \cdots \subseteq \cop(V)
\]
is a hierarchy of inner approximations of~$\cop(V)$ stronger than
the~$\kpcone_r(V)$ hierarchy; it was proposed in the finite setting by Peña,
Vera, and Zuluaga~\cite{PenaVZ2007} and then extended to the infinite setting by
Kuryatnikova and Vera~\cite{Kuryatnikova2019}.

Given a locally independent graph~$G = (V, E)$, set
\[
  \begin{optprob}
    \xi_r(G)=\sup&\langle A, J\rangle\\
    &\int_V A(x, x)\, d\omega(x) = 1,\\
    &A(x, y) = 0\quad\text{if~$xy \in E$,}\\
    &\onerow{\text{$A \in \vzcone_r^*(V)$ is continuous and positive
        semidefinite.}}
  \end{optprob}
\]
This gives a hierarchy of bounds for the measurable independence number, namely
\[
  \xi_1(G) \geq \xi_2(G) \geq \cdots \geq \alpha_\omega(G),  
\]
that is at least as strong as the~$\gamma_r$ hierarchy
of~\S\ref{sec:cop-hierarchy}.  In particular, the same convergence results
of~\S\ref{sec:convergence} hold for this hierarchy.

%=====================================================================

\subsection{Invariant slice-positive functions}%
\label{sec:invariant-slice}

Let~$\Gamma$ be a group acting on a topological space~$V$.  For
each~$v \in V^k$, let~$\langle v\rangle$ be an arbitrary representative of the
orbit~$[v]$ of~$v$ under the action of~$\Gamma$.

Suppose~$F \in C(V^{k+2})$ is slice-positive and $\Gamma$-invariant.  Consider
the function~$K\colon (V^k / \Gamma) \times V^2 \to \R$ such that
\[
  K([v], x, y) = F(x, y, \langle v\rangle)
\]
and for every orbit~$[v]$ let~$K_{[v]}(x, y) = K([v], x, y)$; note that~$K$
depends on the choice of representatives.  The kernel~$K_{[v]}$ is continuous
and positive semidefinite for every orbit~$[v]$.  Moreover, since~$F$ is
$\Gamma$-invariant,~$K_{[v]}$ is invariant under the stabilizer
subgroup~$\stab{\Gamma}{\langle v\rangle}$ of~$\langle v\rangle$.

If we equip~$V^k / \Gamma$ with the quotient topology, then we may ask
whether~$K$ given above is continuous.  This is the case as long as the function
mapping each orbit~$[v]$ to its representative~$\langle v\rangle$ is continuous.

Conversely, say~$K\colon (V^k / \Gamma) \times V^2 \to \R$ is a continuous
function such that~$K_{[v]}$ is a positive-semidefinite
$\stab{\Gamma}{\langle v\rangle}$-invariant kernel for every orbit~$[v]$.  Then
we may define a function~$F\colon V^{k+2} \to \R$ by letting
\[
  F(x, y, v) = K([v], \sigma x, \sigma y)
\]
for any~$\sigma \in \Gamma$ such that~$\sigma v = \langle v\rangle$.  Note
that~$F$ is well defined: if~$\tau v = \langle v\rangle$,
then~$\sigma\tau^{-1} \langle v\rangle = \langle v\rangle$, and so from the
invariance of~$K_{[v]}$ we get
$K([v], \tau x, \tau y) = K([v], \sigma x, \sigma y)$.

By construction,~$F$ is slice-positive and $\Gamma$-invariant.  Indeed,
slice-positivity is clear.  As for invariance, given~$\sigma \in \Gamma$,
let~$\tau \in \Gamma$ be such that~$\tau\sigma v = \langle v\rangle$.  Then
$F(\sigma x, \sigma y, \sigma v) = K([v], \tau\sigma x, \tau\sigma y) = F(x, y,
v)$, as we wanted.  Moreover, if the~$\sigma$s can be chosen to vary
continuously, then~$F$ is also continuous.  More precisely, if there is a
continuous function~$s\colon V^k \to \Gamma$ such
that~$s(v) v = \langle v\rangle$ for every~$v \in V^k$, then~$F$ is continuous.

Our aim is to compute something like~$\xi_1(G_n)$,
where~$G_n = G(S^{n-1}, \{0\})$ is Witsenhausen's graph.  We see now how we can
use the Bachoc-Vallentin kernels~\eqref{eq:bachocv-kernel} to get a subset
of~$\vzcone_1(S^{n-1})$ and therefore a superset of~$\vzcone_1^*(S^{n-1})$,
obtaining in this way a relaxation of~$\xi_1(G_n)$.

If~$Z \in \vzcone_1(S^{n-1})$, then there is a continuous
slice-positive~$F\colon (S^{n-1})^3 \to \R$ such that
$\rey{\Scal_3}(Z \otimes \one - F) \geq 0$.  If~$Z$ is $\orto(n)$-invariant,
then we can assume that also~$F$ is $\orto(n)$-invariant, otherwise we simply
take~$\rey{\orto(n)}F$ instead, which is a continuous slice-positive function.

There is only one orbit for the action of the orthogonal group on the sphere; we
pick the north pole~$e \in S^{n-1}$ as its representative.  So the invariant
function~$F$ is continuous and slice-positive if and only if there is a
continuous, positive-semidefinite, and $\stab{\orto(n)}{e}$-invariant
kernel~$K\colon S^{n-1} \times S^{n-1} \to \R$ such that
$F(x, y, z) = K(Tx, Ty)$, where~$T$ is any orthogonal matrix such that~$Tz = e$.
So the value of~$F(x, y, z)$ depends only on~$e\cdot Tx = x\cdot z$,
$e\cdot Ty = y\cdot z$, and~$Tx\cdot Ty = x\cdot y$.

The Bachoc-Vallentin kernels~\eqref{eq:bachocv-kernel} are positive-semidefinite
and $\stab{\orto(n)}{e}$-invariant.  Fix an integer~$d \geq 1$.  Say~$Z$ is the
$\orto(n)$-invariant kernel given by
\begin{equation}%
  \label{eq:Z-def}
  Z(x, y) = \sum_{k=0}^{2d} f(k) P_k^n(x\cdot y).
\end{equation}
Let~$\overline{Y}_{k,d}^n = \rey{\Scal_3}Y_{k,d}^n$ be the matrix obtained
from~$Y_{k,d}^n$ of~\eqref{eq:Y-matrix} by averaging over all permutations
of~$(u, v, t)$.

If there are positive-semidefinite matrices~$F_k \in \R^{(d-k+1)\times(d-k+1)}$
for~$k = 0$, \dots,~$d$ such that
\begin{equation}%
  \label{eq:strict-pos-constraint}
  \sum_{k=0}^{2d} f(k) (1/3)(P_k^n(u) + P_k^n(v) + P_k^n(t)) - \sum_{k=0}^d
  \langle F_k, \overline{Y}_{k,d}^n(u, v, t)\rangle \geq 0
\end{equation}
for all~$(u, v, t) \in \Delta = \{\, (x\cdot z, y\cdot z, x\cdot y) : \text{$x$,
  $y$, $z \in S^{n-1}$}\,\}$, then~$Z \in \vzcone_1(S^{n-1})$.

Indeed, for~$x$, $y$, $z \in S^{n-1}$ with~$u = x\cdot z$, $v = y\cdot z$,
and~$t = x\cdot y$ we have
\[
  \rey{\Scal_3}(Z \otimes \one)(x, y, z) = \sum_{k=0}^{2d} f(k) (1/3)(P_k^n(u) +
  P_k^n(v) + P_k^n(t)).
\]
The function~$F$ given by
\[
  F(x, y, z) = \sum_{k=0}^d \langle F_k, Y_{k,d}^n(u, v, t)\rangle
\]
is slice-positive and continuous and  
\[
  (\rey{\Scal_3} F)(x, y, z) = \sum_{k=0}^d \langle F_k, \overline{Y}_{k,d}^n(u,
  v, t)\rangle.
\]
Putting it all together, we see that~$Z \in \vzcone_1(S^{n-1})$.

Note that the left side of~\eqref{eq:strict-pos-constraint} is a
polynomial~$p \in \R[u, v, t]$ of degree at most~$2d$ that should be nonnegative
on~$\Delta$.  The polynomial~$p$ is also invariant under the action of~$\Scal_3$
which permutes the variables.  The domain~$\Delta$ is also invariant
under~$\Scal_3$; it is a semi-algebraic set:
$\Delta = \{\, (u, v, t) \in \R^3 : \text{$g_i(u, v, t) \geq 0$ for~$i = 1$,
  \dots,~$4$}\,\}$, where~$g(w) = 1-w^2$ and
\[
\begin{aligned}
  g_1 &= g(u) + g(v) + g(t),&g_2 &= g(u) g(v) + g(u) g(t) + g(v) g(t),\\
  g_3 &= g(u) g(v) g(t),&g_4 &= 1 + 2uvt - u^2 - v^2 - t^2.
\end{aligned}
\]
(See Lemma~3.1 in Machado and Oliveira~\cite{MachadoO2018}.)  So if there are
sums-of-squares polynomials~$q_0$, \dots,~$q_4 \in \R[u, v, t]$ such that
\begin{equation}%
  \label{eq:p-sos}
  p = q_0 + g_1 q_1 + g_2 q_2 + g_3 q_3 + g_4 q_4,
\end{equation}
then~$p$ is nonnegative on~$\Delta$.  Moreover, since~$p$ and the~$g_i$ are all
invariant under~$\Scal_3$, we may assume without loss of generality that
the~$q_i$ are also invariant.

Let~$V_r$ be the matrix indexed by all monomials on~$u$, $v$, and~$t$ of degree
at most~$\lfloor r/2\rfloor$ such that~$(V_r)(m_1, m_2) = m_1 m_2$ for any two
such monomials.  Note that every entry of~$V_r$ is a polynomial of degree at
most~$r$.  A polynomial~$q$ of degree~$2k$ is a sum of squares if and only if
there is a positive-semidefinite matrix~$Q$ such
that~$q = \langle Q, V_{2k}\rangle$.

Using this equivalence and by restricting the degrees of the polynomials~$q_i$
appearing in~\eqref{eq:p-sos}, we can write a sufficient condition for~$p$ to be
nonnegative on~$\Delta$ in terms of positive-semidefinite matrices.  Namely, if
there are positive-semidefinite matrices~$F_k$ and~$Q_i$ such that
\begin{multline}%
  \label{eq:Z-sos}
  \sum_{k=0}^{2d} f(k) (1/3)(P_k^n(u) + P_k^n(v) + P_k^n(t)) - \sum_{k=0}^d
  \langle F_k, \overline{Y}_{k,d}^n(u, v, t)\rangle\\
  = \langle Q_0, V_{2d}\rangle + \langle Q_1, g_1 V_{2d-2}\rangle + \langle Q_2,
  g_2 V_{2d-4}\rangle\\
  + \langle Q_3, g_3 V_{2d-6}\rangle + \langle Q_4, g_4  V_{2d-3}\rangle,
\end{multline}
then~$Z$ given in~\eqref{eq:Z-def} is in~$\vzcone_1(S^{n-1})$.  This leads us to
the definition of the following cone for every fixed~$d$:
\begin{equation}%
  \label{eq:Q-cone-d}
  \begin{split}
    \vzcone_1^d = \{\, (f(0), \ldots, f(2d), 0, \ldots) \in \R^\N :\ &
    \text{there are positive-semidefinite matrices}\\
    &\text{$F_k$ and~$Q_i$ such that~\eqref{eq:Z-sos} holds}\,\}.
  \end{split}
\end{equation}

We were careful to describe the domain~$\Delta$ with invariant polynomials so we
could assume that all polynomials~$q_i$ are likewise invariant.  This can be
used to simplify~\eqref{eq:Z-sos} so we can work with block-diagonal
positive-semidefinite matrices~$Q_i$.  The original idea was presented by
Gatermann and Parrilo~\cite{GatermannP2004}; see also Machado and
Oliveira~\cite{MachadoO2018} and Leijenhorst and De Laat~\cite[\S4]{LaatL2022}
for more recent descriptions of the method and an application to this exact
situation.  This use of symmetry to reduce the problem size is essential to
reach high degrees.

%=====================================================================

\subsection{Witsenhausen's problem: setup}

Our goal is to compute an upper bound for Witsenhausen's
number~$\alpha_n = \alpha_\omega(G_n) / \omega_n$,
where~$G_n = G(S^{n-1}, \{0\})$ is Witsenhausen's graph.  We do so by combining
the cone~$\vzcone_1^d$ with constraints from the Boolean quadratic polytope,
which for a finite set~$V$ is defined as
\[
  \bqp(V) = \conv\{\, x x^\tp : x \in \{0,1\}^V\,\}.
\]
Such constraints were used before by DeCorte, Oliveira, and
Vallentin~\cite{DeCorteOV2022}.  The bound we compute is thus a hybrid
between~$\xi_1(G_n)$ and the bound by DeCorte, Oliveira, and Vallentin.

Given a measurable independent set~$I \subseteq S^{n-1}$ of~$G_n$, let
$A = \rey{\orto(n)}(\chi_I \otimes \chi_I)$.  Then:
\begin{enumerate}
\item[(i)] $A$ is an $\orto(n)$-invariant continuous kernel (by
  Lemma~\ref{lem:tensor-reynolds});

\item[(ii)] $A(x, y) = 0$ for all orthogonal~$x$, $y \in S^{n-1}$;

\item[(iii)] $A$ is positive semidefinite and~$A \in \vzcone_r^*(S^{n-1})$ for
  all~$r \geq 1$;

\item[(iv)] $\bigl(A(x, y)\bigr)_{x,y \in U} \in \bqp(U)$ for every finite~$U
  \subseteq S^{n-1}$;

\item[(v)] $\int_{S^{n-1}} A(x, x)\, d\omega(x) = \omega(I)$
  and~$\langle A, J\rangle = \omega(I)^2$.
\end{enumerate}

We use Schoenberg's theorem to express~$A$ in terms of Jacobi polynomials as
in~\eqref{eq:schoenberg}, so
\[
  A(x, y) = \sum_{k=0}^\infty a(k) P_k^n(x\cdot y)
\]
for some sequence~$a \geq 0$.  Recall that we normalize the polynomials
so~$P_k^n(1) = 1$; together with the addition
formula~(\S\ref{sec:dist-convergence}) this gives
\begin{equation}%
  \label{eq:witsenhausen-norm}
  \int_{S^{n-1}} A(x, x)\, d\omega(x) = \omega_n \sum_{k=0}^\infty
  a(k)\qquad\text{and}\qquad \langle A, J\rangle = \omega_n^2 a(0),
\end{equation}
where~$\omega_n = \omega(S^{n-1})$.

Let~$U \subseteq S^{n-1}$ be a finite set and let~$L\in \R^{U\times U}$
and~$\beta \in \R$ be such that $\langle L, X\rangle \leq \beta$ for
all~$X \in \bqp(U)$.  Then defining~$r\colon \N \to \R$ by
\begin{equation}%
  \label{eq:r-from-L}
  r(k) = \sum_{x, y \in U} L(x, y) P_k^n(x\cdot y)
\end{equation}
we have
\begin{equation}%
  \label{eq:bqp-ineq}
  \sum_{k=0}^\infty a(k) r(k) \leq \beta.
\end{equation}
We call~$(r, \beta)$ a \defi{$\bqp(S^{n-1})$-inequality} and we call the points
in~$U$ the \defi{support points} of the inequality.

Finally, if~$(f(0), \ldots, f(2d), 0, \ldots) \in \vzcone_1^d$,
where~$\vzcone_1^d$ is defined in~\eqref{eq:Q-cone-d}, and if $Z(x, y) =
\sum_{k=0}^{2d} f(k) P_k^n(x\cdot y)$, then~$Z \in \vzcone_1(S^{n-1})$ and
from~\eqref{eq:schoenberg-product} we get
\[
  \sum_{k=0}^{2d} (a(k) / h_k^n) f(k) = \omega_n^{-2} \langle A, Z\rangle \geq
  0,
\]
that is,~$k \mapsto a(k) / h_k^n$ belongs to~$(\vzcone_1^d)^*$.

Let~$(r_1, \beta_1)$, \dots,~$(r_N, \beta_N)$ be any
$\bqp(S^{n-1})$-inequalities and fix some integer~$d \geq 1$.  Put together, our
developments lead us to the following optimization problem, whose optimal value
gives an upper bound for~$\alpha_n$:
\begin{equation}%
  \label{opt:wit-primal}
  \begin{optprob}
    \sup&\sum_{k=0}^\infty a(k)\\[2pt]
    &\sum_{k=0}^\infty a(k) P_k^n(0) = 0,\\[2pt]
    &\sum_{k=0}^\infty a(k) r_i(k) \leq \beta_i\quad\text{for~$i = 1$,
      \dots,~$N$,}\\[1ex]
    &\begin{pmatrix}
      1&\omega_n\sum_{k=0}^\infty a(k)\\[1ex]
      \omega_n\sum_{k=0}^\infty a(k)&\omega_n^2 a(0)
    \end{pmatrix}\text{ is positive semidefinite},\\[1.5eM]
    &\text{$a \geq 0$ and~$k\mapsto a(k) / h_k^n \in (\vzcone_1^d)^*$}.
  \end{optprob}
\end{equation}

The~$2\times 2$ matrix comes from~\eqref{eq:witsenhausen-norm} and~(v) and
is used to normalize the problem.  Note moreover that the objective function is
divided by~$\omega_n$, ensuring that we get a bound
for~$\alpha_n = \alpha_\omega(G_n) / \omega_n$.  Finally, our problem has
infinitely many variables~$a$, but only the first~$2d+1$ of them appear in the
cone constraint with~$(\vzcone_1^d)^*$.  Contrast this with the situation of
the $3$-point bound from~\S\ref{sec:3-pt-bound}.

The dual of this problem is:
\begin{equation}%
  \label{opt:wit-dual}
  \begin{optprob}
    \inf&z_{11} + \sum_{i=1}^N y_i \beta_i\\[2pt]
    &\onerow{\lambda + \sum_{i=1}^\infty y_i r_i(0) - \omega_n z_{12} -
      \omega_n^2
      z_{22} - f(0) \geq 1,}\\[2pt]
    &\lambda P_k^n(0) + \sum_{i=1}^N y_i r_i(k) - \omega_n z_{12} - f(k) \geq
    1&\text{for
      all~$k = 1$, \dots,~$2d$,}\\[2pt]
    &\lambda P_k^n(0) + \sum_{i=1}^N y_i r_i(k) - \omega_n z_{12} \geq
    1&\text{for all~$k
      \geq 2d+1$},\\[1ex]
    &\onerow{\begin{pmatrix}
        z_{11}&z_{12}/2\\
        z_{12}/2&z_{22}
      \end{pmatrix}\text{ is positive semidefinite},}\\[1eM]
    &\onerow{\text{$y \geq 0$ and $k\mapsto h_k^n f(k) \in \vzcone_1^d$.}}
  \end{optprob}
\end{equation}
It is easy to show that the objective value of any feasible solution
of~\eqref{opt:wit-dual} is greater or equal than the objective value of any
feasible solution of~\eqref{opt:wit-primal}.  So any feasible solution of the
dual gives an upper bound for~$\alpha_n$.  Note moreover that the constraint
``$k\mapsto h_k^n f(k) \in \vzcone_1^d$'' is expressed in terms
of~\eqref{eq:Z-sos}, namely we require there to be positive-semidefinite
matrices~$F_k$ and~$Q_i$ such that
\begin{multline}%
  \label{eq:Z-sos-final}
  \sum_{k=0}^{2d} f(k) (h_k^n/3)(P_k^n(u) + P_k^n(v) + P_k^n(t)) - \sum_{k=0}^d
  \langle F_k, \overline{Y}_{k,d}^n(u, v, t)\rangle\\
  - \langle Q_0, V_{2d}\rangle - \langle Q_1, g_1 V_{2d-2}\rangle - \langle Q_2,
  g_2 V_{2d-4}\rangle\\
  - \langle Q_3, g_3 V_{2d-6}\rangle - \langle Q_4, g_4  V_{2d-3}\rangle = 0.
\end{multline}
So~\eqref{opt:wit-dual} is a semidefinite programming problem with finitely many
variables but infinitely many constraints.

%=====================================================================

\subsection{Witsenhausen's problem: solution and verification}

To solve~\eqref{opt:wit-dual} we use the \texttt{ClusteredLowRankSolver} of
Leijenhorst and De Laat~\cite{LaatL2022}; the input for the solver is generated by
a Julia program.  The program and all data files used are available in the
Harvard Dataverse repository~\cite{Oliveira2023}.

To find good $\bqp(S^{n-1})$-inequalities, we use a separation heuristic
described by DeCorte, Oliveira, and Vallentin~\cite{DeCorteOV2022}.  The
inequalities used are also included in the repository and need not be
recomputed.

Table~\ref{tab:detailed} contains a detailed account of all the bounds computed
from~\eqref{opt:wit-dual}.  Solving the problem for~$d = 14$ and~$18$ takes time
and memory, and so files with the corresponding solutions are also available in
the repository.

\begin{table}[t]
  \begin{center}
    \small
    \begin{tabular}{ccccccc}
      &                  &\multicolumn{2}{c}{\textsl{Previous upper
                           bound}}&\multicolumn{3}{c}{\textsl{New upper bound}}\\
      $n$&\textsl{Lower bound}&\textsl{Simple}&\textsl{Best}&$d$&\textsl{No
                                                                  BQP}&\textsl{With
                                                                        BQP}\\
      \rowcolor{gray!20}
      3 & 0.2928\ldots  & 0.3333\ldots  & 0.30153 & 6 & 0.316925 & 0.300708\\
      \rowcolor{gray!20}
      &  &  &  & 10 & 0.309298 & 0.298998\\
      \rowcolor{gray!20}
      &  &  &  & 14 & 0.305627 & 0.298341\\
      \rowcolor{gray!20}
      &  &  &  & 18 & 0.303294 & 0.297742\\
      4 & 0.1816\ldots  & 0.25  & 0.21676 & 6 & 0.223633 & 0.207617\\
      &  &  &  & 10 & 0.211825 & 0.199402\\
      &  &  &  & 14 & 0.205479 & 0.196162\\
      &  &  &  & 18 & 0.201445 & 0.194297\\
      \rowcolor{gray!20}
      5 & 0.1161\ldots  & 0.2  & 0.16765  & 6 & 0.167357 & 0.151541\\
      \rowcolor{gray!20}
      &  &  &  & 10 & 0.153819 & 0.141539\\
      \rowcolor{gray!20}
      &  &  &  & 14 & 0.146612 & 0.137142\\
      \rowcolor{gray!20}
      &  &  &  & 18 & 0.142349 &  0.134588\\
      6 & 0.0755\ldots  & 0.1666\ldots  & 0.13382  & 6 & 0.130829 & 0.116599\\
      &  &  &  & 10 & 0.116509 & 0.105200\\
      &  &  &  & 14 & 0.109989 & 0.100374\\
      &  &  &  & 18 & 0.106727 & 0.098095\\
      \rowcolor{gray!20}
      7 & 0.0498\ldots  & 0.1428\ldots  & 0.11739 & 6 & 0.106059 & 0.093031\\
      \rowcolor{gray!20}
      &  &  &  & 10 & 0.091477 & 0.081221\\
      \rowcolor{gray!20}
      &  &  &  & 14 & 0.086656 & 0.077278\\
      \rowcolor{gray!20}
      &  &  &  & 18 & 0.084787 & 0.075751\\
      8 & 0.0331\ldots  & 0.125 & 0.09981 & 6 & 0.088750 & 0.076801\\
      &  &  &  & 10 & 0.074309 & 0.064919\\
      &  &  &  & 14 & 0.071676 & 0.063287\\
      &  &  &  & 18 & 0.070607 & 0.061178\\
    \end{tabular}
  \end{center}
  \bigskip
  
  \caption{Lower and upper bounds for Witsenhausen's parameter~$\alpha_n$.  The
    lower bound is given by the double-cap conjecture.  The simple upper bound
    was given by Witsenhausen~\cite{Witsenhausen1974} and is just~$1/n$; the
    best previous upper bounds are by DeCorte, Oliveira, and
    Vallentin~\cite{DeCorteOV2022}.  The table gives upper bounds obtained from
    solving~\eqref{opt:wit-dual} with and without $\bqp(S^{n-1})$-inequalities
    and for several values of~$d$.}%
  \label{tab:detailed}
\end{table}

Since~\eqref{opt:wit-dual} has infinitely many linear constraints, to solve it
we select some finite set~$S \subseteq \{0, 1, \ldots\}$ and consider only the
constraints for~$k \in S$.  After a solution is found it has to be verified,
that is, we need to check that all constraints are indeed satisfied.

Let~$(\lambda, y, z, f, F, Q)$ be a candidate solution to~\eqref{opt:wit-dual},
where~$F$ and~$Q$ are as in~\eqref{eq:Z-sos-final}, returned by the solver.  The
first step is to certify ourselves that~$f$, $F$, and~$Q$ indeed
satisfy~\eqref{eq:Z-sos-final}.

This is certainly not true: the solver works with floating-point arithmetic, and
so~\eqref{eq:Z-sos-final} will not hold.  Rather, the left side
of~\eqref{eq:Z-sos-final} will be a polynomial with coefficients close to~0.
Since the \texttt{ClusteredLowRankSolver} uses high-precision floating-point
arithmetic, the coefficients will be quite small; let~$\eta$ be the largest
absolute value of any such coefficient.

It is always possible to perturb the matrices~$Q_i$ in order to satisfy the
constraint; the order of the perturbation depends on~$\eta$.  We want to do so
and keep the~$Q_i$ positive semidefinite, and as long as the minimum eigenvalues
of the matrices~$Q_i$ are large enough compared to~$\eta$, this is always
possible.  The solutions stored in the repository have large minimum
eigenvalues, several orders of magnitude larger than~$\eta$, and so this
perturbation of the~$Q_i$ can always be carried out.  Note that we do not have
to actually change the~$Q_i$; it suffices to know that such a perturbation is
possible, since then we know we can get a feasible solution if we want to.  This
procedure was used before by De Laat, Oliveira, and Vallentin~\cite{LaatOV2014}.

Checking that the linear constraints for all~$k \geq 0$ are satisfied is more
difficult; we use the approach outlined in DeCorte, Oliveira, and
Vallentin~\cite{DeCorteOV2022}.

The idea is as follows.  Let~$\lhs(k)$ be the left side of the~$k$th linear
constraint in~\eqref{opt:wit-dual} and
write~$\lhs(\infty) = \lim_{k \to \infty} \lhs(k)$; we will see that this limit
exists.  Then we follow the steps:
\begin{enumerate}
\item[(i)] As long as~$z_{22} > 0$, we can change~$z_{11}$ and~$z_{12}$ to
  get~$\lhs(\infty) \geq 1 + \eta$ for some~$\eta > 0$.  The more we
  change~$z_{12}$, the more we have to change~$z_{11}$, and the worse the bound
  gets.

\item[(ii)] Next, for some~$\epsilon < \eta$ we find a~$k_0$ such
  that~$|\lhs(k) - \lhs(\infty)| \leq \epsilon$ for all~$k \geq k_0$.
  Then~$\lhs(k) \geq \lhs(\infty) - \epsilon \geq 1 + \eta - \epsilon > 1$, and
  so all constraints are satisfied for~$k \geq k_0$.

\item[(iii)] Finally, we check the constraints for~$k = 0$, \dots,~$k_0$, and by
  changing~$z$ again we can make all these constraints satisfied.
\end{enumerate}

If we choose our initial sample~$S$ well, then all constraints will be almost
satisfied, and we will not have to change~$z$ too much in order to get a
feasible solution.  This is the procedure implemented by the
\texttt{fix\_linear\_constraints} function in the Julia program in the
repository~\cite{Oliveira2023}.

Let us see the details of the procedure.  The asymptotic formula for the Jacobi
polynomials~\cite[Theorem~8.21.8]{Szego1975} implies that~$P_k^n(t) \to 0$
as~$k \to \infty$ for all~$t \in (-1, 1)$.  We make sure that all the
$\bqp(S^{n-1})$-inequalities~\eqref{eq:bqp-ineq} we use have support points~$U$
such that distinct~$x$, $y \in U$ have inner product~$x\cdot y$ bounded away
from~$\pm 1$.  So if~$r$ is given as in~\eqref{eq:r-from-L}, then
\[
  r(\infty) = \lim_{k\to\infty} r(k) = \trace L
\]
and
\begin{equation}%
  \label{eq:rk-infty-diff}
  |r(k) - r(\infty)| \leq \sum_{\substack{x,y \in U\\x \neq y}} |L(x, y)|
  |P_k^n(x\cdot y)|.
\end{equation}

We also have
\[
  \lhs(\infty) = \sum_{i=1}^N y_i r_i(\infty) - \omega_n z_{12}.
\]
Given~$\epsilon > 0$ we want to get~$k_0$ as in~(ii).  Note that
\[
  |\lhs(k) - \lhs(\infty)| \leq |\lambda| |P_k^n(0)| + \sum_{i=1}^N y_i |r_i(k)
  - r_i(\infty)|.
\]
Fix~$k_0$.  Using~\eqref{eq:rk-infty-diff}, we see that to find an upper bound
for the left side above for all~$k \geq k_0$, it suffices to find for
all~$k \geq k_0$ an upper bound on~$|P_k^n(t)|$ for~$t = 0$ and all
other~$t \in (-1, 1)$ that occur as inner products between distinct support
points of the $\bqp(S^{n-1})$-inequalities we use.

To do so rigorously, we use an integral representation for the ultraspherical
polynomials due to Gegenbauer (take~$\lambda = (n - 2)/2$ in Theorem~6.7.4 from
Andrews, Askey, and Roy~\cite{AndrewsAR1999}):
\[
  P_k^n(\cos\theta) = R(n)^{-1} \int_0^\pi F(\phi)^k
  \sin^{n-3}\phi\, d\phi,
\]
where
\[
  F(\phi) = \cos\theta + i\sin\theta\cos\phi\qquad\text{and}\qquad
  R(n) = \int_0^\pi \sin^{n-3} \phi\, d\phi.
\]

Note that~$|F(\phi)|^2 = \cos^2\theta + \sin^2\theta \cos^2\phi$, so
\[
  |P_k^n(\cos\theta)| \leq R(n)^{-1} \int_0^\pi (\cos^2\theta + \sin^2\theta
  \cos^2\phi)^{k/2} \sin^{n-3}\phi\, d\phi.
\]
The right side is decreasing in~$k$, and we can estimate the integrals rigorously
using interval arithmetic.

For~(iii) we need to compute~$\lhs(k)$ for all~$k \leq k_0$.  We would like to
do this rigorously, using for instance interval arithmetic.  The most
time-consuming step here is to compute the~$r_i$ functions.  In practice, this
step involves evaluating the polynomials~$P_k^n$ for values of~$k$ that can
exceed $100{,}000$.

The Jacobi polynomials~$P_k^n$ are given by a simple recurrence, namely
\[
  P_k^n(u) = a_k^n(u) P_{k-1}^n(u) + b_k^n P_{k-2}^n(u)
\]
for~$k \geq 2$ with~$P_1^n(u) = u$ and~$P_0^n(u) = 1$, where
\[
  a_k^n(u) = \frac{2k + 2\alpha - 1}{k + 2\alpha} u\qquad\text{and}\qquad b_k =
  -\frac{k-1}{k + 2\alpha}
\]
with~$\alpha = (n-3) / 2$.  This recurrence comes from formula~(4.5.1) in
Szegö~\cite{Szego1975}, adapted to our normalization of~$P_k^n(1) = 1$.

The recurrence is very stable: even using double-precision floating-point
arithmetic it is possible to accurately evaluate the polynomial for very high
degrees for points in~$[-1, 1]$.  If we use this recurrence with interval
arithmetic though, the error estimation quickly gets out of hand: if
$\lim_{k\to\infty} a_k(u) > 1$, then the error bound grows exponentially.

Using interval arithmetic then requires very high precision and is very slow,
though not prohibitively so.  In any case, we can trust floating-point
computations.  Using the recurrence amounts to solving by backward substitution
a linear system with a triangular matrix whose entries are the numbers
$a_k(u)$,~$b_k$, and~$1$, and this matrix is well conditioned, so the error we
make in solving the system is very small.  The error was analyzed for instance
by Barrio~\cite{Barrio2002}.  The Julia program that performs the verification
uses high-precision floating-point arithmetic.

%%%%%%%%%%%%%%%%%%%%%%%%%%%%%%%%%%%%%%%%%%%%%%%%%%%%%%%%%%%%%%%%%%%%%%

\section*{Acknowledgments}

We would like to thank David de Laat, Nando Leijenhorst, Fabrício Caluza
Machado, and Willem de Muinck Keizer for fruitful discussions.  David de Laat
and Nando Leijenhorst also gave some much needed technical support regarding the
\texttt{ClusteredLowRankSolver}.  The optimization problems were solved in a
computational cluster at TU Delft maintained by Joffrey Wallaart.

%%%%%%%%%%%%%%%%%%%%%%%%%%%%%%%%%%%%%%%%%%%%%%%%%%%%%%%%%%%%%%%%%%%%%%

%%%%%%%%%%%%%%%%%%%%%%%%%%%%%%%%%%%%%%%%%%%%%%%%%%%%%%%%%%%%%%%%%%%%%% 
%                                                                    %
% Appendix                                                           %
%                                                                    %
%%%%%%%%%%%%%%%%%%%%%%%%%%%%%%%%%%%%%%%%%%%%%%%%%%%%%%%%%%%%%%%%%%%%%%

\appendix

\section{Pólya's theorem and a proof of Theorem~\ref{thm:olga-juan}}%
\label{apx:cop-hierarchy}

Let~$V$ be a finite set.  A \defi{symmetric $k$-tensor} is a function~$T\colon
V^k \to \R$ invariant under the action of~$\Scal_k$, that is, invariant under
permutations of the variables.  Symmetric tensors correspond to homogeneous
polynomials: if~$[n] = \{1, \ldots, n\}$ and if~$T\colon [n]^k \to \R$ is a
symmetric $k$-tensor, then the polynomial
\begin{equation}%
  \label{eq:tensor-poly}
  p(w) = \langle T, w^{\otimes k}\rangle,
\end{equation}
where~$w \in \R^n$ and where~$\langle \cdot, \cdot\rangle$ is the Euclidean
inner product in the space of symmetric $k$-tensors, is an $n$-variable
homogeneous polynomial of degree~$k$.  Conversely, if~$p$ is an $n$-variable
homogeneous polynomial of degree~$k$, then there is a unique symmetric
$k$-tensor~$T$ such that~\eqref{eq:tensor-poly} holds.

Let~$p \in \R[w]$, where~$w = (w_1, \ldots, w_n)$, be a homogeneous polynomial.
If for some integer~$r \geq 0$ all the coefficients of the polynomial
\begin{equation}%
  \label{eq:polya}
  (w_1 + \cdots + w_n)^r p(w) = (\one^\tp w)^r p(w)
\end{equation}
are nonnegative, then~$p(w) \geq 0$ for all~$w \geq 0$.  Pólya's theorem gives a
converse: if~$p(w) > 0$ for all nonzero~$w \geq 0$, then there is~$r$ such that
all coefficients of~\eqref{eq:polya} are nonnegative.

Let~$V$ be a finite set and say~$A\colon V^2 \to \R$ is a matrix in the
algebraic interior of the copositive cone.  This means in particular that the
homogeneous polynomial~$p(w) = w^\tp A w$ is positive for all
nonzero~$w \geq 0$.  By Pólya's theorem, there is~$r$ such that all coefficients
of the polynomial~$(\one^\tp w)^r (w^\tp A w)$ are nonnegative.  Now note that
\[
  (\one^\tp w)^r (w^\tp A w) = \langle A \otimes \one^{\otimes r},
  w^{\otimes(r+2)}\rangle = \langle \rey{\Scal_{r+2}}(A \otimes \one^{\otimes
    r}), w^{\otimes(r+2)}\rangle,
\]
so the coefficients of our polynomial are given by the symmetric $(r+2)$-tensor
$\rey{\Scal_{r+2}}(A \otimes \one^{\otimes r})$, which is therefore nonnegative.
It follows that~$A \in \kpcone_r(V)$, proving Theorem~\ref{thm:olga-juan} for
finite~$V$.

For a function~$T\colon V^k \to \R$ and any set~$U \subseteq V$, we denote
by~$T[U]$ the restriction of~$T$ to~$U^k$.  To prove Theorem~\ref{thm:olga-juan}
for infinite~$V$, we use the following infinite-dimensional version of Pólya's
theorem, which is interesting in itself.

\begin{theorem}%
  \label{thm:polya-gen}
  Let~$V$ be a compact Hausdorff space equipped with a Radon measure~$\omega$
  that is positive on open sets.  Given a symmetric~$T \in C(V^k)$, let
  $M = \max\{\, |T(v)| : v \in V^k\,\}$
  and~$\lambda = \inf\{\, \langle T, f^{\otimes k}\rangle :
  \text{$f \in L^2(V)$, $f \geq 0$, and $\langle \one, f\rangle = 1$}\,\}$.
  If~$\lambda > 0$, then for every~$r > k(k-1)M/(2\lambda) - k$ we
  have~$\rey{\Scal_{r+k}}(T \otimes \one^{\otimes r}) \geq 0$.
\end{theorem}

\begin{proof}
Let~$U \subseteq V$ be any finite nonempty set.  Given a
function~$w\colon U \to \Rplus$ with~$\one^\tp w = 1$, fix~$\epsilon > 0$ and
let~$P(x)$ for~$x \in U$ be disjoint sets of positive measure such
that~$x \in P(x)$ for all~$x \in U$ and such that~$T$ varies at most~$\epsilon$
in any set of the form~$P(x_1) \times \cdots \times P(x_k)$ for~$x_1$,
\dots,~$x_k \in U$.  Such sets always exist: see for instance Theorem~4.4 in
DeCorte, Oliveira, and Vallentin~\cite{DeCorteOV2022}, where this is shown
for~$k = 2$, though the proof immediately generalizes for~$k \geq 3$; the proof
of Lemma~\ref{lem:nonneg-tensor} uses the same idea.

Set~$f = \sum_{x \in U} w(x) \omega(P(x))^{-1} \chi_{P(x)}$; note
that~$f \geq 0$ and~$\langle\one, f\rangle = 1$.  For~$x_1$, \dots,~$x_k \in U$,
write~$P(x_1, \dots, x_k) = P(x_1) \times \cdots \times P(x_k)$.  We have
\[
  \begin{split}
    \langle T, f^{\otimes k}\rangle &= \sum_{x_1, \ldots, x_k \in U}
    \biggl(\prod_{i=1}^k w(x_i) \omega(P(x_i))^{-1}\biggr) \int_{P(x_1, \ldots, x_k)} T(y_1,
    \ldots, y_k)\, d\omega(y)\\
    &\leq \sum_{x_1, \ldots, x_k \in U}
    \biggl(\prod_{i=1}^k w(x_i) \omega(P(x_i))^{-1}\biggr) \int_{P(x_1, \ldots, x_k)} T(x_1,
    \ldots, x_k) + \epsilon\, d\omega(y)\\
    &=\langle T[U], w^{\otimes k}\rangle + \epsilon,
\end{split}
\]
hence by taking~$\epsilon \to 0$ we see
that~$\langle T[U], w^{\otimes k}\rangle \geq \langle T, f^{\otimes k}\rangle
\geq \lambda$.  This holds for every~$U$ and~$w$.

Fix a finite nonempty set~$U \subseteq V$.  We now apply a theorem of Powers and
Reznick~\cite[Theorem~1]{PowersR2001} to the homogeneous
polynomial~$p(w) = \langle T[U], w^{\otimes k}\rangle$ of degree~$k$.
Since~$p(w) \geq \lambda > 0$ for all~$w \geq 0$ with~$\one^\tp w = 1$, the
theorem of Powers and Reznick says that all the coefficients of the polynomial
$(\one^\tp w)^r p(w)$ are positive for all~$r > k(k-1)M'/(2\lambda) - k$,
where~$M' = \max\{\, |T(v)| : v \in U^k\,\}$.  This means that all the entries
of~$\rey{\Scal_{r+k}}(T[U] \otimes \one^{\otimes r})$ are nonnegative.
Since~$U$ is any finite subset and since~$M \geq M'$, we are done.
\end{proof}

\begin{proof}[Proof of Theorem~\ref{thm:olga-juan}]
If~$A$ is in the algebraic interior of~$\copc(V)$, then there is~$\lambda > 0$
such that~$A - \lambda J$ is copositive.  So if~$f \in L^2(V)$
is such that~$f \geq 0$ and~$\langle \one, f\rangle = 1$, then
\[
  0 \leq \langle A - \lambda J, f \otimes f\rangle = \langle A, f \otimes
  f\rangle - \lambda,
\]
and so~$\langle A, f \otimes f\rangle \geq \lambda$.  Using
Theorem~\ref{thm:polya-gen} we see that there is~$r$ such that
$\rey{\Scal_{r+2}}(A \otimes \one^{\otimes r}) \geq 0$, and
so~$A \in \kpcone_r(V)$, proving the theorem.
\end{proof}


\begin{thebibliography}{46}
\bibitem{AndrewsAR1999}
G.E.~Andrews, R.~Askey, and R.~Roy, {\it Special Functions},
Encyclopedia of Mathematics and its Applications~71, Cambridge
University Press, Cambridge, 1999.

\bibitem{BachocNOV2009}
C.~Bachoc, G.~Nebe, F.M.~de Oliveira Filho, and F.~Vallentin, Lower
bounds for measurable chromatic numbers, {\it Geometric and Functional
Analysis\/}~19 (2009) 645--661.

\bibitem{BachocV2008}
C.~Bachoc and F.~Vallentin, New upper bounds for kissing numbers from
semidefinite programming, {\it Journal of the American Mathematical
Society\/}~21 (2008) 909--924.

\bibitem{Barrio2002}
R.~Barrio, Rounding error bounds for the Clenshaw and Forsythe
algorithms for the evaluation of orthogonal polynomial series, {\it
Journal of Computational and Applied Mathematics\/}~138 (2002)
185--204.

\bibitem{Bochner1941}
S.~Bochner, Hilbert distances and positive definite functions, {\it
Annals of Mathematics\/}~42 (1941) 647--656.

\bibitem{CohnE2003}
H.~Cohn and N.~Elkies, New upper bounds on sphere packings I, {\it
Annals of Mathematics\/}~157 (2003) 689--714.

\bibitem{CohnKMRV2017}
H.~Cohn, A.~Kumar, S.D.~Miller, D.~Radchenko, and M.~Viazovska, The
sphere packing problem in dimension~24, {\it Annals of
Mathematics\/}~185 (2017) 1017--1033.

\bibitem{CohnLS2022}
H.~Cohn, D.~de Laat, and A.~Salmon, Three-point bounds for sphere
packing, arXiv:2206.15373, 2022, 37pp.

\bibitem{CohnS2021}
H.~Cohn and A.~Salmon, Sphere packing bounds via rescaling,
arXiv:2108.10936, 2021, 43pp.

\bibitem{DeCorteOV2022}
E.~DeCorte, F.M.~de Oliveira Filho, and F.~Vallentin, Complete
positivity and distance-avoiding sets, {\it Mathematical Programming,
Series~A\/}~191 (2022) 487--558.

\bibitem{DelsarteGS1977}
P.~Delsarte, J.M.~Goethals, and J.J.~Seidel, Spherical codes and
designs, {\it Geometriae Dedicata\/}~6 (1977) 363--388.

\bibitem{DobreDFV2016}
C.~Dobre, M.E.~Dür, L.~Frerick, and F.~Vallentin, A copositive
formulation for the stability number of infinite graphs, {\it
Mathematical Programming, Series~A\/}~160 (2016) 65--83.

\bibitem{Folland1995}
G.B.~Folland, {\it A Course in Abstract Harmonic Analysis}, Studies in
Advanced Mathematics, CRC Press, Boca Raton, 1995.

\bibitem{Folland1999}
G.B.~Folland, {\it Real Analysis: Modern Techniques and Their
Applications\/} (Second Edition), John Wiley \& Sons, Inc., New York,
1999.

\bibitem{FranklW1981}
P.~Frankl and R.M.~Wilson, Intersection theorems with geometric
consequences, {\it Combinatorica\/}~1 (1981) 357--368.

\bibitem{GatermannP2004}
K.~Gatermann and P.A.~Parrilo, Symmetry groups, semidefinite programs,
and sums of squares, {\it Journal of Pure and Applied Algebra\/}~192
(2004) 95--128.

\bibitem{GrotschelLS1988}
M.~Grötschel, L.~Lovász, and A.~Schrijver, {\it Geometric Algorithms
and Combinatorial Optimization}, Algorithms and Combinatorics~2,
Springer-Verlag, Berlin, 1988.

\bibitem{GvozdenovicLV2009}
N.~Gvozdenović, M.~Laurent, and F.~Vallentin, Block-diagonal
semidefinite programming hierarchies for 0/1 programming, {\it
Operations Research Letters\/}~37 (2009) 27--31.

\bibitem{Handel2000}
D.~Handel, Some homotopy properties of spaces of finite subsets of
topological spaces, {\it Houston Journal of Mathematics\/}~26 (2000)
747--764.

\bibitem{Kalai2015}
G.~Kalai, Some old and new problems in combinatorial geometry~I:\@
around Borsuk's problem, in: {\it Surveys in combinatorics~2015},
London Mathematical Society Lecture Note Series~424, Cambridge
University Press, Cambridge, 2015, pp.~147--174.

\bibitem{KlerkP2007}
E.~de Klerk and D.V.~Pasechnik, A linear programming reformulation of
the standard quadratic optimization problem, {\it Journal of Global
Optimization\/}~37 (2007) 75--84.

\bibitem{KlerkP2002a}
E.~de Klerk and D.V.~Pasechnik, Approximation of the stability number
of a graph via copositive programming, {\it SIAM Journal on
Optimization\/}~12 (2002) 875--892.

\bibitem{Kuryatnikova2019}
O.~Kuryatnikova, {\it The many faces of positivity to approximate
structured optimization problems}, Ph.D. Thesis, Tilburg University,
2019.

\bibitem{Laat2020}
D.~de Laat, Moment methods in energy minimization: new bounds for
Riesz minimal energy problems, {\it Transactions of the American
Mathematical Society\/}~373 (2020) 1407--1453.

\bibitem{LaatMOV2022}
D.~de Laat, F.C.~Machado, F.M.~de Oliveira Filho, and F.~Vallentin,
$k$-point semidefinite programming bounds for equiangular lines, {\it
Mathematical Programming, Series~A\/}~194 (2022) 533--567.

\bibitem{LaatMM2022}
D.~de Laat, W.~de Muinck Keizer, and F.C.~Machado, The Lasserre
hierarchy for equiangular lines with a fixed angle, arXiv:2211.16471,
2022, 25pp.

\bibitem{LaatOV2014}
D.~de Laat, F.M.~de Oliveira Filho, and F.~Vallentin, Upper bounds for
packings of spheres of several radii, {\it Forum of Mathematics,
Sigma\/}~2 (2014) 42pp.

\bibitem{LaatV2015}
D.~de Laat and F.~Vallentin, A semidefinite programming hierarchy for
packing problems in discrete geometry, {\it Mathematical Programming,
Series~B\/}~151 (2015) 529--553.

\bibitem{Lasserre2001}
J.B.~Lasserre, Global optimization with polynomials and the problem of
moments, {\it SIAM Journal on Optimization\/}~11 (2001) 796--817.

\bibitem{LaatL2022}
N.~Leijenhorst and D.~de Laat, Solving clustered low-rank semidefinite
programs arising from polynomial optimization, {\it Mathematical
Programming Computation\/}~16 (2024) 503--534.

\bibitem{Lovasz1979}
L.~Lovász, On the Shannon capacity of a graph, {\it IEEE Transactions
on Information Theory\/}~IT-25 (1979) 1--7.

\bibitem{MachadoO2018}
F.C.~Machado and F.M.~de Oliveira Filho, Improving the semidefinite
programming bound for the kissing number by exploiting polynomial
symmetry, {\it Experimental Mathematics\/}~27 (2018) 362--369.

\bibitem{Montina2011}
A.~Montina, Communication cost of classically simulating a quantum
channel with subsequent rank-1 projective measurement, {\it Physical
Review A\/}~84 (2011) 4pp.

\bibitem{MotzkinS1965}
T.S.~Motzkin and E.G.~Straus, Maxima for graphs and a new proof of a
theorem of Turán, {\it Canadian Journal of Mathematics\/}~17 (1965)
533--540.

\bibitem{Oliveira2023}
F.M.~de Oliveira Filho, Data files and code related to the upper bound
for Witsenhausen's parameter, {\tt
https://doi.org/10.7910/DVN/TWI1SR}, Harvard Dataverse, 2023, V1.

\bibitem{Oliveira2009}
F.M.~de Oliveira Filho, {\it New bounds for Geometric Packing and
Coloring via Harmonic Analysis and Optimization}, Ph.D. Thesis,
University of Amsterdam, 2009.

\bibitem{OliveiraV2013}
F.M.~de Oliveira Filho and F.~Vallentin, A quantitative version of
Steinhaus's theorem for compact, connected, rank-one symmetric spaces,
{\it Geometriae Dedicata\/}~167 (2013) 295--307.

\bibitem{OliveiraV2010}
F.M.~de Oliveira Filho and F.~Vallentin, Fourier analysis, linear
programming, and densities of distance-avoiding sets in~$\R^n$, {\it
Journal of the European Mathematical Society\/}~12 (2010) 1417--1428.

\bibitem{PenaVZ2007}
J.~Peña, J.~Vera, and L.F.~Zuluaga, Computing the stability number of
a graph via linear and semidefinite programming, {\it SIAM Journal on
Optimization\/}~18 (2007) 87--105.

\bibitem{PowersR2001}
V.~Powers and B.~Reznick, A new bound for Pólya's theorem with
applications to polynomials positive on polyhedra, {\it Journal of
Pure and Applied Algebra\/}~164 (2001) 221--229.

\bibitem{Schoenberg1942}
I.J.~Schoenberg, Positive definite functions on spheres, {\it Duke
Mathematical Journal\/}~9 (1942) 96--108.

\bibitem{Simon2011}
B.~Simon, {\it Convexity: An Analytic Viewpoint}, Cambridge Tracts in
Mathematics~187, Cambridge University Press, Cambridge, 2011.

\bibitem{Szego1975}
G.~Szegö, {\it Orthogonal Polynomials\/} (Fourth Edition), American
Mathematical Society Colloquium Publications Volume~XXIII, American
Mathematical Society, Providence, 1975.

\bibitem{Viazovska2017}
M.S.~Viazovska, The sphere packing problem in dimension~8, {\it Annals
of Mathematics\/}~185 (2017) 991--1015.

\bibitem{Wang1952}
H-C.~Wang, Two-point homogeneous spaces, {\it Annals of
Mathematics\/}~55 (1952) 177--191.

\bibitem{Witsenhausen1974}
H.S.~Witsenhausen, Spherical sets without orthogonal point pairs, {\it
American Mathematical Monthly\/}~10 (1974) 1101--1102.

\end{thebibliography}
\end{document}